\documentclass[a4paper, UKenglish, thm-restate, autoref]{lipics-v2021}

\usepackage{mathpartir}
\usepackage{anyfontsize}
\usepackage{prftree}
\usepackage{hyperref}
\usepackage{tikz-cd}
\usepackage{mathtools}
\usepackage[bb=dsserif]{mathalpha}
\usepackage{bm}
\usepackage{stmaryrd}

\title{The biequivalence of path categories and axiomatic Martin-L\"of type theories}

\author
    {Dani\"el Otten}
    {ILLC, University of Amsterdam, The Netherlands \and \url{https://www.otten.co/}}
    {daniel@otten.co}
    {https://orcid.org/0000-0003-2557-3959}
    {This publication is part of the project “The Power of Equality” (with project number OCENW.M20.380) of the research programme Open Competition Science M 2 which is (partly) financed by the Dutch Research Council (NWO).}
\author
    {Matteo Spadetto}
    {DMIF, University of Udine, Italy \and LS2N, University of Nantes, France \and \url{https://spadetto.github.io/}}
    {matteo.spadetto.42@gmail.com}
    {https://orcid.org/0000-0002-6495-7405}
    {The author was supported by an E-COST-CA20111 Short-Term Scientific Mission Grant --- that funded his research visit to Amsterdam in Spring 2023, where and when this research started --- a School of Mathematics full-time EPSRC Doctoral Training Partnership Studentship 2019/2020, the Italian MUR PRIN 2022 “STENDHAL”, and  eOTP RECIPROG.}

\authorrunning{D. Otten and M. Spadetto}

\Copyright{Dani\"el Otten and Matteo Spadetto}



\hideLIPIcs

\acknowledgements{We are thankful to Benno van den Berg, Rafa\"el Bocquet, Jacopo Emmenegger, Ambrus Kaposi, Marino Miculan, and Niels van der Weide for many helpful discussions and ideas.}

\nolinenumbers

\EventEditors{Stefano Guerrini and Barbara K\"{o}nig}
\EventNoEds{2}
\EventLongTitle{34th EACSL Annual Conference on Computer Science Logic (CSL 2026)}
\EventShortTitle{CSL 2026}
\EventAcronym{CSL}
\EventYear{2026}
\EventDate{February 23--28, 2026}
\EventLocation{Paris, France}
\EventLogo{}
\SeriesVolume{363}
\ArticleNo{37}

\SetSymbolFont{stmry}{bold}{U}{stmry}{m}{n}

\newcommand\disp{\mathrel{\hspace{-0.175pt}-\raisebox{0.8pt}{\(\scriptstyle\hspace{-1.2pt}\triangleright\)}}}
\newcommand\fib{\mathrel{\hspace{-0.175pt}-\raisebox{0.8pt}{\(\scriptstyle\hspace{-1.2pt}\triangleright\hspace{-0.8pt}\triangleright\)}}}
\newcommand\tto{\mathrel{\smash{\ooalign{\raisebox{3pt}{\(\scriptstyle\sim\)}\cr\(\to\)}}}}
\newcommand\tdisp{\mathrel{\smash{\ooalign{\raisebox{3pt}{\(\scriptstyle\sim\)}\cr\(\disp\)}}}}
\newcommand\tfib{\mathrel{\smash{\ooalign{\raisebox{3pt}{\(\scriptstyle\sim\)}\cr\(\fib\)}}}}
\newcommand\iso{\mathrel{\smash{\ooalign{\raisebox{3pt}{\(\scriptstyle\sim\)}\cr\(\to\)\cr\raisebox{-2pt}{\hspace{-0.175pt}\scalebox{0.85}[1]{\(-\)}}}}}}

\tikzcdset{pullback/.style = {"\lrcorner"{anchor = center, pos = 0.125}, draw = none}}
\tikzcdset{disp/.style     = {-{Triangle[open]}}}
\tikzcdset{fib/.style      = {-{Triangle[open].Triangle[open]}}}
\tikzcdset{equiv/.style    = {"\raisebox{3pt} {\(\scriptstyle\sim\)}" {sloped, marking}}}
\tikzcdset{equiv'/.style   = {"\raisebox{-1.5pt}{\(\scriptstyle\sim\)}" {sloped, marking}}}
\tikzcdset{iden/.style     = {"\raisebox{-2pt}{\scalebox{0.85}[1]{\(-\)}}" {sloped, marking}}}
\tikzcdset{iden'/.style    = {"\raisebox{2pt} {\scalebox{0.85}[1]{\(-\)}}" {sloped, marking}}}
\tikzcdset{tfib/.style     = {equiv, fib}}
\tikzcdset{tfib'/.style    = {equiv', fib}}
\tikzcdset{iso/.style      = {equiv, iden}}
\tikzcdset{iso'/.style     = {equiv', iden'}}

\newcommand\from\leftarrow
\newcommand\To\Rightarrow
\newcommand\From\Leftarrow
\newcommand\longto\longrightarrow
\newcommand\longfrom\longleftarrow
\newcommand\limplies\rightarrow
\newcommand\limpliedby\leftarrow
\newcommand\liff\leftrightarrow
\newcommand\biglor\bigvee
\newcommand\bigland\bigwedge
\newcommand\<\langle
\renewcommand\>\rangle
\newcommand\llb\llbracket
\newcommand\rrb\rrbracket

\newcommand\bb\mathbb
\renewcommand\cal\mathcal
\renewcommand\sf\mathsf
\renewcommand\frak\mathfrak

\newcommand\W{\rm W}

\newcommand\ctx{~\sf{ctx}}
\newcommand\type{~\sf{type}}
\newcommand\Id{\sf{Id}}
\newcommand\refl{\sf{refl}}
\newcommand\pair{\sf{pair}}
\newcommand\ind{\sf{ind}}

\newcommand\weak\triangledown
\newcommand\liltriangle\weak

\usetikzlibrary{arrows.meta}
\usetikzlibrary{decorations.pathmorphing}
\tikzset{|/.tip = {Bar[width = .8ex, round]}}

\makeatletter

\renewcommand\vdots{\vbox{\baselineskip4\p@ \lineskiplimit\z@ \hbox{.}\hbox{.}\hbox{.}}}

\newlength\makesize@widthfrom
\newlength\makesize@widthto
\newlength\makesize@widthhalfdiff
\newcommand\makesize[3][c]{%
    \ifmmode%
        \text{\makesize@do{#1}{\(\m@th{#2}\)}{\(\m@th{#3}\)}}%
    \else%
        \makesize@do{#1}{#2}{#3}%
    \fi%
}
\newcommand\makesize@do[3]{%
    \settowidth\makesize@widthfrom{#2}%
    \settowidth\makesize@widthto{#3}%
    \setlength\makesize@widthhalfdiff{0.5\makesize@widthto - 0.5\makesize@widthfrom}%
    \if#1s%
        \makebox[\makesize@widthto][s]{#2}%
    \else%
        \vphantom{#3}%
        \if#1l\smash{#2}\fi%
        \hspace*{\makesize@widthhalfdiff}%
        \if#1c\smash{#2}\fi%
        \hspace*{\makesize@widthhalfdiff}%
        \if#1r\smash{#2}\fi%
    \fi%
}

\newcommand\prfstack{\prf@init\prf@Nstack}
\newcommand\prf@Nstack{\@ifnextchar[{\prf@stackoptions}{\prf@stackrule}}
\def\prf@stackoptions[#1]{\prf@opts#1,]\prf@Nstack}
\def\prf@stackrule{%
    \edef\prf@rulechoice{\ifprf@rule\else{}\fi}%
    \expandafter\prf@stacklabel\prf@rulechoice%
}
\def\prf@stacklabel#1{%
    \def\prf@hackpar{#1}%
    \edef\prf@labelchoice{{\noexpand\prf@hackpar}\ifprf@label\else{}\fi}%
    \expandafter\prf@stackinit\prf@labelchoice%
}
\newcommand\prf@stackinit[3]{%
    \ifprf@label\raisebox{0.46\baselineskip}{\text{#2}}\hspace*\prflabelskip\fi%
    \begin{array}[b]{@{}l@{}l@{}}%
    \hspace*\prflineextra%
    \@ifnextchar\bgroup%
        {\prf@stackrec{#1}{#2}{#3}}%
        {\prf@stackrec{#1}{#2}{}{#3}}%
}
\newcommand\prf@stackrec[4]{\@ifnextchar\bgroup%
    {\prf@stackrec{#1}{#2}{%
        #3%
        \ifprfSTRUToption\mathstrut\fi%
        \hspace*\prflineextra\\%
        \hspace*\prflineextra%
        #4%
    }}{%
        #3%
        \ifprfSTRUToption\mathstrut\fi%
        \hspace*\prflineextra\\%
        \noalign{\hrule height\prflinethickness\vspace*\prflinepadafter}%
        \hspace*\prflineextra%
        #4%
        \hspace*\prflineextra%
        \end{array}%
        \ifprf@rule\hspace*\prfrulenameskip\raisebox{0.46\baselineskip}{\text{#1}}\fi%
    }%
}

\makeatother

\newcommand\graybf[1]{\textcolor{lipicsGray}{\sffamily\bfseries\mathversion{bold}#1}}

\begin{document}
    \maketitle
    
    \begin{abstract}
        The semantics of extensional type theory has an elegant categorical description: models of extensional \(=\)-types, \(\bb1\)-types, and \(\Sigma\)-types are biequivalent to finitely complete categories, while adding \(\Pi\)-types yields locally Cartesian closed categories.
        We establish parallel results for axiomatic type theory, which includes systems like cubical type theory, where the computation rule of the \(=\)-types only holds as a propositional axiom instead of a definitional reduction.
        In particular, we prove that models of axiomatic \(=\)-types, and standard \(\bb1\)- and \(\Sigma\)-types are biequivalent to certain path categories, while adding axiomatic \(\Pi\)-types yields dependent homotopy exponents.

        This biequivalence simplifies axiomatic =-types, which are more intricate than extensional ones since they permit higher dimensional structure.
        Specifically, path categories use a primitive notion of equivalence instead of a direct reproduction of the syntactic elimination rules and computation axioms.
        We apply our correspondence to prove a coherence theorem: we show that these weak homotopical models can be turned into equivalent strict models of axiomatic type theory.
        In addition, we introduce a more modular notion, that of a display map path category, which only models axiomatic \(=\)-types by default, while leaving room to add other axiomatic type formers such as \(\bb1\)-, \(\Sigma\)-, and \(\Pi\)-types.
    \end{abstract}

    \section{Introduction and motivation}
    \label{sec:introduction}
    
    Dependent type theory, and most notably \emph{intensional type theory (ITT)}, has become a cornerstone of logic, with applications in proof assistants, programming languages, and the foundations of mathematics.
    Models of type theory form categories, and categorical formulations \cite{phdcartmell,moggi_1991,MR1201808} provide a powerful tool to understand the behaviour of these systems.
    However, giving a satisfactory categorical formulation of dependent type theory is challenging.
    First of all, we want our requirements to be simple categorical properties instead of direct reproductions of the syntax; and secondly, categorical structure usually only satisfies stability conditions up to isomorphism instead of the strict equality imposed in the syntax.
    This means that it is significantly easier to produce \emph{weak models}, which only satisfy stability up to equivalence.
    To deal with the second issue, strictification techniques \cite{hofmannlccc,MR2461866,MR3372323,MR4481908,MR4550392,kaposi2025type} have been developed to turn weakly stable structure into strictly stable structure.
    This paper focuses on the first issue: providing a categorical and homotopy theoretic presentation of the weak models of dependent type theory.
    We combine our results with existing techniques for the second issue to turn our weak models into \textit{(strict) models}.
    
    One difficulty in studying ITT is the fact that it has two notions of equality: \emph{propositional equality (=)}, which is an internal equality type, and \emph{definitional equality (\(\equiv\))}, which is an external judgment.
    Distinguishing these notions provides ITT with its good computational properties: definitional equality is a decidable fragment of propositional equality and allows type checking to be algorithmically decidable.
    This fragment can be changed, such as by allowing the addition of rewrite rules \cite{Cockx2020rewrite}.
    The two extremes are: \emph{extensional type theory (ETT)} \cite{martin1982constructive}, where every equality is definitional, and \emph{axiomatic type theory (ATT)} \cite{2021arXiv210200905V,spadetto25}, where no equality is definitional.
    
    For ETT the external and internal notions of equality match.
    This prevents it from being compatible with homotopy type theory: it proves uniqueness of identity proofs (UIP) which contradicts univalence \cite{Hofmann1997,HoTTBook}.
    In addition, definitional equality is undecidable in ETT, which means that type checking is undecidable as well.
    By contrast, the semantics of ETT is easier to explore and allows for a simple categorical description of its type formers.
    Models can be formulated using locally Cartesian closed categories as seen by Seely \cite{MR0727078}.
    However, Seely interprets substitutions using pullbacks, which are only specified up to isomorphism, leading to coherence problems.
    The precise statement is a biequivalence of 2-categories as shown by Clairambault and Dybjer \cite{MR3272793} building on work of Curien, Hofmann and Garner \cite{curien1993substitution,hofmannlccc,CURIEN201499}, which has been extended by Maietti \cite{MAIETTI2009319} and Van der Weide \cite{nielsvdweide}: \[\hfill\begin{array}{l|r}
        \text{Syntax} & \text{Semantics} \\
        \hline
        \text{extensional \(=\), \(\bb1\), \(\Sigma\)} & \text{finitely complete categories} \\
        \text{extensional \(=\), \(\bb1\), \(\Sigma\), \(\Pi\)} & \text{locally Cartesian closed categories}
    \end{array}\hfill\]
    In this paper we generalise this to ATT.
    The main reason to study axiomatic type theory is that --- while it only has one notion of equality like ETT, and is therefore easier to study -- it is minimal and therefore compatible with ITT and its extensions, and provides a larger class of models.
    In addition, there are conservativity results for ATT which means that we can use it to study ITT and ETT: for example, ETT is conservative over ATT extended with function extensionality and UIP \cite{Winterhalter2020, boulierwinterhalter, conservativity}. Moreover, from a mathematical perspective, axiomatic type formers enjoy the advantage of being \textit{homotopy invariant}, as explained in \cite{MR3050430,MR3614859}: any type that is homotopy equivalent to the one given by the formation rule of a type former satisfies the same rules.
    
    We generalise the known correspondence by considering the notion of a path category introduced by Van den Berg and Moerdijk \cite{MR3795638,MR3828037}.
    Path categories have equivalences as a primitive notion, which allows the omission of explicit interpretations of the elimination and computation rules.
    This is a considerable simplification, since it is exactly these rules that become unwieldy in ATT, especially in a setting without \(\Pi\)-types.
    We show a correspondence:
    \[\hfill\begin{array}{l|r}
        \text{Syntax} & \text{Semantics} \\
        \hline
        \text{extensional \(\bb1\), \(\Sigma\) and axiomatic \(=\)} & \text{path categories} \\
        \text{extensional \(\bb1\), \(\Sigma\) and axiomatic \(=\), \(\Pi\) } & \text{ path cats with dependent homotopy exponents}
    \end{array}\hfill\]
    This extends the existing results since finitely complete categories can be seen as degenerate path categories where dependent homotopy exponentials simplify to dependent exponentials.
    We do this by proving an isomorphism of 2-categories between path categories and an existing notion of weak semantics: display map categories.
    Then we apply the results of Lumsdaine and Warren \cite{MR3372323} and Bocquet \cite{MR4481908} to strictify those path categories that satisfy the additional logical framework (\ref{def:LF}) condition, giving a biequivalence with models satisfying this condition.
    We also introduce a new and more fine-grained notion --- that of a \emph{display map path category}, which distinguishes types and telescopes --- and prove similar correspondences: 
    \[\hfill\begin{array}{l|r}
        \text{Syntax} & \text{Semantics} \\
        \hline
        \text{axiomatic \(=\), \((\bb1)\), \((\Sigma)\), \((\Pi)\)} & \text{suitable display map path categories}
    \end{array}\hfill\]
    This result is more modular without introducing significant additional complexity: display map path categories only have axiomatic =-types by default and can be extended with any combination of \(\bb1\), \(\Sigma\), and \(\Pi\) in a relatively simple way.

    \subparagraph*{Our contribution.} Our main contribution is twofold: (a) providing a categorical formulation of the weak models of axiomatic \(=\)-types, and exploiting it to obtain a coherence result for the class of path categories, and (b) a characterisation of the strict models of axiomatic type theory in terms of path categories.
    In detail, we show that the 2-category of path categories is isomorphic to a 2-category of suitable display map categories.
    We consider display map categories to phrase the semantics of type theory because they form a good middle ground: they allow for a formulation of type formers that closely follows the syntax, while being expressive enough to consider both weak and strict models.
    Additionally, they are relatively simple and close to the homotopical notions.
    
    The idea of formulating the semantics of axiomatic \(=\)-types using path categories originates from the work of Van den Berg and Moerdijk \cite{MR3795638,MR3828037}.
    Our contribution lies in clarifying precisely in what sense path categories serve as models for these axiomatic \(=\)-types: every path category is \textit{rooted} --- that is, every context is essentially build by extending the empty context with finitely many types --- and models axiomatic \(=\)-types, and extensional \(\bb1\)- and \(\Sigma\)-types.
    We show that, when they are \textit{structured} (come equipped with choices of substitutions), and satisfy the \ref{def:LF} condition, they can be strictified into actual models.
    
    Additionally, we introduce a more fine-grained notion, that of a display map path category, which distinguishes between types and telescopes.
    Because of this distinction, they do not model extensional \(\bb1\)- and \(\Sigma\)-types, but still model axiomatic \(=\)-types.
    This makes them a suitable semantics for a minimal ATT --- only axiomatic \(=\)-types --- which can be extended with axiomatic \(\bb1\)-types, \(\Sigma\)-types, and \(\Pi\)-types.

    \subparagraph*{Related work.}
    This paper is positioned at the intersection of two goals: (a) finding a well-structured categorical description of models for ATT, in such a way that there is a modular correspondence \cite{MR2189827, MR2525957} between certain categories, where syntax is not explicitly articulated, and specific type theories, and (b) proving coherence theorems.
    From this perspective, our work is in line with that of Maietti \cite{MR2189827} and that of Clairambault and Dybjer \cite{MR3272793}. Maietti’s work is based on characterising models of certain theories via universal properties. This is largely due to the fact that her work focuses on extensional theories.
    To extend this to axiomatic theories, we need to use a weaker notion of universal properties in the form of \textit{homotopy (co)limits}. 
    Such an approach seems necessary for inductive types such as the natural numbers type and $\W$-types.
    For $=$-types, $\bb1$-types, and $\Sigma$-types, we show that it suffices to ask the constructor to be an equivalence: a homotopy universal property already follows from this requirement.
    Another approach consists in describing axiomatic type formers by means of higher dimensional universal properties in higher dimensional structures, as done in \cite{spadetto25}.
    We also reference Van der Weide's work \cite{nielsvdweide}, which generalizes both Maietti's and Clairambault and Dybjer's results to univalent categories.
    
    The correspondence that we establish is similar to the one by Clairambault and Dybjer \cite{MR3272793}, since finitely complete categories are precisely the path categories in which the only equivalences are the isomorphisms and every morphism is a fibration.
    However, a formulation of our result as a generalization of the former is hard to find.
    On the one hand, finitely complete categories induce \textit{pseudo-stable} type formers in the corresponding display map category, which is enough to obtain a split display map category with stable type formers, hence a genuine model, by \textit{right-adjoint} splitting.
    On the other hand, a general path category only provides the corresponding display map category with \textit{weakly stable} type formers, which forces us to base our splitting strategy on the \textit{left-adjoint} splitting functor, as described by Lumsdaine and Warren \cite{MR3372323} and Bocquet \cite{MR4481908}.
    However, the latter requires the input display map category to be structured, which in general does not follow from our hypotheses.
    Hence, we cannot hope to link general path categories to models of ATT, but only to weak models with weakly stable structure, whereas Clairambault and Dybjer's biequivalence automatically involves pseudo-stable structure --- and hence falls within the domain of the right-adjoint splitting.
    Our solution consists in restricting our biequivalence between path categories and weak models to the \textit{structured}, or \textit{cloven}, case in both sides, so that the latter fall within the domain of the left-adjoint splitting: in this way we deduce a biequivalence between cloven path categories and strict models of our type theory.

    \subparagraph*{Structure of the paper.} In \autoref{sec:syntax} we describe \emph{axiomatic \(=\)-types} syntactically, and explain why the based version is preferred in a setting without \(\Pi\)-types.
    In \autoref{sec:dispcat}, we review the notion of a display map category as well as the semantic notion of axiomatic \(=\)-types, and explain \textit{strict} and \textit{weak} stability in both the general and the structured case.
    Analogously, \autoref{sec:pathcat} recalls the notion of a path category together with its basic properties.
    
    In \autoref{sec:iso} we describe how path categories extend the notion of a \textit{clan} with based axiomatic \(=\)-types.
    We use this fact to show that path categories are rooted display map categories with weakly stable based axiomatic \(=\)-types and extensional \(\bb1\)- and \(\Sigma\)-types. Secondly, we prove that these conditions on a display map category are themselves enough to recover the structure of a path category, and that this correspondence extends to an isomorphism between the corresponding 2-categories.
    In \autoref{sec:disppathcat}, we introduce the notion of a display map path category adding a distinction between display maps and fibrations, and prove a similar 2-isomorphism.
    Consequence of our correspondence are explored in \autoref{sec:strictify}, where we strictify to obtain genuine models.

    \section{Syntax}\label{sec:syntax}

    We will assume some familiarity with the syntax of Martin-L\"of type theory; for an overview, see \cite{martin1984intuitionistic,HoTTBook}.
    Our main focus will be a minimal type theory with only propositional equality (the \(=\)-types) and no definitional equality (the \(\equiv\)-judgement).
    For the structural rules, the unit type (\(\bb1\)), dependent product types (\(\Sigma\)) and dependent function types (\(\Pi\)), see \autoref{app:syntaxsemantics}.
    Type formers in intensional Martin-L\"of type theory are given by formation, introduction, elimination, and computation (\(\beta\)) rules.
    For the \(=\)-types, these are:  \[
        \prfstack[r]{\(=\)F,}
            {\Gamma\vdash A\type \qquad \Gamma\vdash a:A \qquad \Gamma\vdash a':A}
            {\Gamma\vdash a=_Aa'\type}\qquad
        \prfstack[r]{\(=\)I,}
            {\Gamma\vdash a:A}
            {\Gamma\vdash\refl(a):a=_Aa} \] \[
        \prfstack[r]{\(=\)E,}
            {\Gamma,x:A,x':A,\chi:x=_Ax'\vdash C\type}
            {\Gamma,x:A\vdash d:C[x/x',\refl(x)/\chi]}
            {\Gamma\vdash\alpha:a=_Aa'}
            {\Gamma\vdash\sf{ind}^=_d(\alpha):C[a/x,a'/x',\alpha/\chi]} \qquad
        \prfstack[r]{\(=\)\(\beta\).}
            {\Gamma,x:A,x':A,\chi:x=_Ax'\vdash C\type}
            {\Gamma,x:A\vdash d:C[x/x',\refl(x)/\chi]}
            {\Gamma\vdash a:A}
            {\Gamma\vdash\sf{ind}^=_d(\refl(a))\equiv d[a/x]}
    \]
    In extensional type theory, we also add a uniqueness (\(\eta\)) rule: \begin{align*}
        &\prfstack[r]{\(=\)\(\eta\).}
            {\Gamma,x:A,x':A,\chi:x=_Ax'\vdash c:C}
            {\Gamma\vdash\alpha:a=_Aa'}
            {\Gamma\vdash\sf{ind}^=_{c[x/x',\refl(x)/\chi]}(\alpha)\equiv c[a/x,a'/x',\alpha/\chi]}
    \end{align*}
    The \(\eta\)-rule implies that propositional equality and definitional equality coincide, making type checking undecidable.
    We are mainly interested in the other direction --- a minimal type theory without definitional equality --- so we modify intensional \(=\)-types in two ways: weaken the \(\beta\)-rule (definitional equality) to a \(\beta\)-axiom (propositional equality), and fix the variable \(x\) in place as a constant term \(a\) in the eliminator: \begin{align*}
        &\prfstack[r]{\({=}\rm E_{\rm{based}}\),}
            {\Gamma,x':A,\chi:a=_Ax'\vdash C\type}
            {\Gamma\vdash d:C[a/x',\refl(a)/\chi]}
            {\Gamma\vdash\alpha:a=_Aa'}
            {\Gamma\vdash\sf{ind}^=_d(\alpha):C[a'/x',\alpha/\chi]}\, &
        &\prfstack[r]{\({=}\beta_{\rm{based}}^{\rm{ax}}\).}
            {\Gamma,x':A,\chi:a=_Ax'\vdash C\type}
            {\Gamma\vdash d:C[a/x',\refl(a)/\chi]}
            {\Gamma\vdash\beta^=_d:\sf{ind}^=_d(\refl(a))=d}
    \end{align*}
    These two modifications are orthogonal.
    The main focus of this paper is on the first one: weakening the \(\beta\)-rule to a \(\beta\)-axiom, that is, only requiring the \(=\)-type to be inhabited instead of requiring a judgemental equality.
    This weaker notion of \(=\)-type is called an \textit{axiomatic \(=\)-type}, and is the version that appears in cubical type theory \cite{MR3781068} as well as in axiomatic type theory (type theory without reductions) \cite{2021arXiv210200905V}.
    The other modification, fixing one of the two endpoints, is a technical variation known as \emph{based} path induction or the Paulin-Mohring eliminator \cite{paulin1993inductive}.
    With \(\Pi\)-types, based and unbased path induction are equivalent.
    However, without \(\Pi\)-types, the based version is stronger and better behaved.
    For example, the transport \(\alpha_*b:B[a'/x]\) of a term \(b:B[a/x]\) along an equality \(\alpha:a=_Aa'\) can be defined as \(\ind_b^=(\alpha)\) using based path induction, while it is not definable using unbased path induction.
    
    Alternatively, to give the unbased version the same strength as the based version, we can extend it with a telescope \(\Omega\).
    This is called the \emph{parametrized} or Frobenius version of path induction.
    For example, axiomatic parametrized unbased path induction is given by: \begin{align*}
        &\prfstack[r]{\raisebox\baselineskip{\(=\)E\(_{\rm{param}}\), \({=}\beta_{\rm{param}}^{\rm{ax}}\).}}
            {\Gamma,x:A,x':A,\chi:x=_Ax',\Omega\vdash C\type}
            {\Gamma,x:A,\Omega[x/x',\refl(x)/\chi]\vdash d:C[x/x',\refl(x)/\chi]}
            {\Gamma\vdash\alpha:a=_Aa'}
            {\Gamma,\Omega[a/x,a/x',\alpha/\chi]\vdash\sf{ind}^=_d(\alpha):C[a/x,a'/x',\alpha/\chi] \\ \hspace*\prflineextra
            \Gamma,\Omega[a/x,a/x',\alpha/\chi]\vdash\beta^=_d:\sf{ind}^=_d(\refl(a))=d}
    \end{align*}
    We have equivalences between axiomatic/intensional/extensional versions of path induction: \[\hfill
        \text{unparametrized based }\liff\text{ parametrized based }\liff\text{ parametrized unbased},
    \hfill\] while the original unparametrized unbased version is weaker without \(\Pi\)-types.
    We prove a semantic version of this in \autoref{prop:identityvariations}.

    In general we use the following terminology: a type former is \emph{axiomatic} if we only have a \(\beta\)-axiom, \emph{intensional} if it has a \(\beta\)-rule, and \emph{extensional} if it also has an \(\eta\)-rule.
    Note that the \(\beta\)-axiom generally implies the existence of an \(\eta\)-axiom \cite{MR3050430,MR3614859}.
    The type theory is \emph{axiomatic}/\emph{intensional}/\emph{extensional} if every type former is.
    Note that having extensional =-types implies that every other type former is extensional.

    \section{Display map categories}\label{sec:dispcat}

    There is a big zoo of semantic frameworks for the structural part of dependent type theory; see \cite{Ahrens_Lumsdaine_North_2025} for an overview of the equivalences.
    These range from frameworks that closely follow the syntax (such as categories with families, which have primitive notions for contexts, types, and terms) to frameworks that simplify the structure (such as display map categories, where a type \(A\) in context \(\Gamma\) is identified with the display map or context projection \(\Gamma.A\to\Gamma\), and a term \(a\) of \(A\) is identified with the induced context map \(\Gamma\to\Gamma.A\)).
    We will be working in the later setting for two reasons: it allows us to talk about both strict and weak models, and can easily be extended using ideas of homotopy theory as we will see in \autoref{sec:pathcat}.

    \begin{definition}[Display map category \cite{hyland1989theory,taylor1999practical}]
        A \emph{display map category} is a category \(\cal C\) with a subclass of maps \(\cal C^{\sf{disp}}\subseteq\cal C^\to\) called the \emph{display maps (\(\disp\))} that is replete (closed under isomorphism in the arrow category \(\cal C^\to\)), and closed under pullback along any map.
    \end{definition}
    The objects of a display map category are called \emph{contexts}.
    A \emph{type} \(A\) in context \(\Gamma\) consists of a context suggestively written \(\Gamma.A\) together with a display map \(p_A:\Gamma.A\disp\Gamma\).
    A \emph{term} \(a\) of \(A\) is a section of \(p_A:\Gamma.A\disp\Gamma\), that is, a map \(a:\Gamma\to\Gamma.A\) such that \(p_Aa=1_\Gamma\).
    More generally, a \emph{term} \(a\) of \(A\) for a context map \(\sigma:\Delta\to\Gamma\) is a map \(a:\Delta\to\Gamma.A\) such that \(pa=\sigma\).
    Such a term of \(A\) for \(\sigma\) is equivalent to having a term of the pullback \(A[\sigma]\) of \(A\) along \(\sigma\): \[\hfill\begin{tikzcd}[row sep = 1em]
    	{\Delta.A[\sigma]} & {\Gamma.A} \\
    	\Delta & \Gamma.
    	\arrow[disp, "p_{A[\sigma]}"', from=1-1, to=2-1]
    	\arrow[disp, "p_A", from=1-2, to=2-2]
    	\arrow["\sigma"', from=2-1, to=2-2]
    	\arrow["{\sigma^{\liltriangle}}", from=1-1, to=1-2]
        \arrow["\lrcorner"{anchor=center, pos=0.125}, draw=none, from=1-1, to=2-2]
    \end{tikzcd}\hfill\]
    We call \(A[\sigma]\) the \emph{re-indexing} or \emph{substitution} of \(A\) along \(\sigma\), and \(\sigma^\weak\) the \emph{weakening} of \(\sigma\) for \(A\).
    We extend both of these notions further.
    If $a$ is a term of type $A$, then, for any choice of \(A[\sigma]\), we write $a[\sigma]$ for the unique term of type $A[\sigma]$ such that $(\sigma^\liltriangle)(a[\sigma])=a \sigma$: \[\hfill\begin{tikzcd}[row sep = 1em]
    	{\Delta.A[\sigma]} & {\Gamma.A} \\
    	\Delta \ar[u, bend left = 45, "{a[\sigma]}"] & \Gamma. \ar[u, bend right = 45, "a"', pos = 0.4]
    	\arrow[disp, from=1-1, to=2-1]
    	\arrow[disp, from=1-2, to=2-2]
    	\arrow["\sigma"', from=2-1, to=2-2]
    	\arrow["{\sigma^{\liltriangle}}", from=1-1, to=1-2]
        \arrow["\lrcorner"{anchor=center, pos=0.125}, draw=none, from=1-1, to=2-2]
    \end{tikzcd}\hfill\]
    For a type \(B\) in context \(\Gamma\) we write $B^\liltriangle\coloneqq B[p_A]$ for the induced type in an extended context \(\Gamma.A\); and for a term \(b\) of \(B\) we write \(b^\weak\coloneqq b[p_A]\) for the induced term of \(B^\weak\).
    Finally, for any choice of \(B^\liltriangle\), we use the categorical notation \((a,b):\Gamma\to\Gamma.A.B^\liltriangle\) for the term of the pullback induced by terms \(a:\Gamma\to\Gamma.A\) and \(b:\Gamma\to\Gamma.B\).
    We extend this to the dependent case by writing \((a,b):\Gamma\to\Gamma.A.B\) for the term \(a^\liltriangle b\) induced by terms \(a:\Gamma\to\Gamma.A\) and \(b:\Gamma\to\Gamma.B[a]\).
    For a type \(A\) we have always the \emph{variable} term or \emph{diagonal} \(\delta_a\coloneqq(1_{\Gamma.A},1_{\Gamma.A}):\Gamma.A\to\Gamma.A.A^\weak\).

    \begin{definition}[Root]
        We call a sequence of types \(A_0.\dots.A_{n-1}\), such that \(A_i\) is a type in context \(\Gamma.A_0.\dots.A_{i-1}\), a \emph{telescope} of \(\Gamma\), and a composition of display maps \(\Gamma.A_0.\dots.A_{n-1}\disp\cdots\disp\Gamma\) a \emph{fibration (\(\fib\))}.
        A display map category is called \emph{rooted} if it has a terminal object $1$ (the \emph{empty context} or \emph{root}), and any map \(\Gamma\to1\) is a fibration\footnote{
            Note that having a root is the correct way to state democracy in the absence of extensional \(\bb1\)- and \(\Sigma\)-types.
            Democracy means that every context has a representation, and is usually stated as every context being isomorphic to \(1.A\) for a type \(A\) in the empty context \cite[Definition 2.6]{MR3272793}.
        }.
    \end{definition}
    Type formers follow the syntactic specification.
    For example, an =-type is given by:

    \begin{definition}[\(=\)-type]
        We say that a display map category has \emph{axiomatic (based) $=$-types} if, for every type $A$ in context $\Gamma$, there exists the data of a \(=\)-type for \(A\):
        \begin{description}
            \item[\emph{F.}] a type \(\Id_A\) in context $\Gamma.A.A^\liltriangle$ \emph{(for some choice of re-indexing)},
            \item[\emph{I.}] a term \(\refl_A\) of type \(\Id_A[\delta_A]\) in context \(\Gamma.A\), where \(\delta_A\coloneqq(1_{\Gamma.A},1_{\Gamma.A}):\Gamma.A\to\Gamma.A.A^\liltriangle\): \[\hfill\begin{tikzcd}[row sep = 1.5em, column sep = 5em]
            	\Gamma.A.\Id_A[\delta_A] \ar[d, disp]\ar[r,"\delta_A^\liltriangle"]\ar[dr, "\lrcorner"{anchor=center, pos=0.125}, draw=none] & \Gamma.A.A^\liltriangle.\Id_A \ar[d, disp] \\
                \Gamma.A \ar[r,"\delta_A"']\ar[u, bend left = 45, "\refl_A", dashed] & \Gamma.A.A^\liltriangle\mathrlap,
        	\end{tikzcd}\hfill\]
            \item[\emph{E.}]
            for every term \(a\) of type \(A\), every type \(C\) in context \(\Gamma.A.\Id_A[a^\liltriangle]\), and every term \(d\) of type \(C[a,\refl[a]]\)\;\footnote{
                Note that \(\refl[a]\) is a morphism of the form \(\Gamma\to\Gamma.\Id[\delta_A][a]\) and because \(\delta_Aa=(1,1)a=(a,a)=a^\liltriangle a\) we can also write the codomain as \(\Gamma.\Id[a^\weak][a]\) so that \((a,\refl[a]):\Gamma\to\Gamma.A.\Id_A[a^\weak]\).
                However, a structured display map category (see below) does not have to choose the same object for \(\Gamma.\Id[\delta_A][a]\) and \(\Gamma.\Id[a^\weak][a]\).
                So, there we should understand the statement as leaving a canonical isomorphism between the chosen objects implicit.
                A strict display map category is guaranteed to choose the same re-indexing.
            }, a term \(\ind^=_d\) of \(C\): \[\hfill\begin{tikzcd}[column sep = 5em, row sep = 1.5em]
                \Gamma.C[a,\refl[a]]\ar[d, disp]\ar[r,"{(a,\refl[a])^\liltriangle}"]\ar[dr, "\lrcorner"{anchor=center, pos=0.125}, draw=none] & \Gamma.A.\Id_A[a^\liltriangle].C \ar[d, disp] \\
                \Gamma \ar[r,"{(a,\refl[a])}"']\ar[u, bend left = 45, "d"] & \Gamma.A.\Id_A[a^\liltriangle], \ar[u, bend right = 45, "{\ind_d^=}"', dashed]
        	\end{tikzcd}\hfill\]
            \item[\emph{\(\beta\).}] for every \(a,C,d\) as before, \(\ind^=_d[a,\refl[a]]\) is propositionally equal to \(d\), that is, we have a term \(\beta^=_d\) of type \(\Id_{C[a,\refl[a]]}[\ind^=_d[a,\refl[a]],d]\).
        \end{description}
        We are interested in the axiomatic notion, but for clarity we also explain stronger notions.
        We call an \(=\)-type for a type $A$ \emph{intensional} if it satisfies the \(\beta\)-rule: for every \(C,d\) as before we have that \(\beta^=_d=\refl[d]\), and therefore $\ind^=_d[a,\refl[a]]=d$.
        We call the \(=\)-type \emph{extensional} if it also satisfies the \(\eta\)-rule: for every \(C\) and term \(c\) of \(C\) we have that $\ind^=_{c[a,\refl[a]]}=c$.
    \end{definition}
    We say that a display map category with $=$-types has \emph{weakly stable} \(=\)-types if for every type \(A\) in context \(\Gamma\) there exist an \(=\)-type \(\Id_A\) and for every \(\sigma:\Delta\to\Gamma\) we can equip \(\Id_A[\sigma^{\liltriangle\liltriangle}]\) and \[\hfill\begin{tikzcd}[cramped,column sep=5em]
        {\Delta.A[\sigma]} &
        {\Delta.A[\sigma].\Id_A[\delta_A][\sigma^\liltriangle]} \iso {\Delta.A[\sigma].\Id_A[\sigma^{\liltriangle\liltriangle}][\delta_{A[\sigma]}],}
        \arrow["{\refl_A[\sigma^\liltriangle]}", from=1-1, to=1-2]
    \end{tikzcd}\hfill\] with the additional data required to form an \(=\)-type.
    To explain the intuition for this notion, we first need to explain the idea of an equivalence.
    A \emph{(homotopy) equivalence} between types \(A\) and \(B\) fibered over a context \(\Gamma\) is a term \(b:\Gamma.A\to\Gamma.B\) of \(B\) for \(p_A\) such that there exists a term \(a:\Gamma.B\to\Gamma.A\) of \(A\) for \(p_B\) so that the two compositions are pointwise propositionally equal to the identity, that is, we have terms of \(\Id_A[ab,1_{\Gamma.A}]\) and \(\Id_B[ba,1_{\Gamma.B}]\).
    One can observe that between any two \(=\)-types \(\Id_A,\Id_A'\) for \(A\) there exists an equivalence that respects the data of an =-type, and that any type that is equivalent to an \(=\)-type for \(A\) can be given the additional data to form an \(=\)-type.
    Hence, weak stability is equivalent to the statement that =-types are preserved by re-indexing up to equivalence.
    This certainly holds in the syntax because the substitution of an =-type is still an =-type.
    
    As explained in \cite{Ahrens_Lumsdaine_North_2025}, there exists a 2-category $\sf{DispCat}$ of display map categories, functors preserving display maps and pullbacks of display maps, and natural transformations.
    We can specialize this to 0-cells that are rooted and 1-cells that preserve the terminal object.
    Another specialization is to 0-cells with additional weakly stable type formers and 1-cells that weakly preserve them; see \autoref{app:iso}.
    In \autoref{sec:iso}, we show that the 2-category of path categories is isomorphic to the 2-category \(\sf{DispCat}^{\sf{root}}_{=_{\sf{axi}},\bb1_{\sf{ext}},\Sigma_{\sf{ext}}}\) of rooted display map categories with weakly stable axiomatic \(=\)-types, extensional \(\bb1\)-types, and extensional \(\Sigma\)-types.

    \subsection{Strict display map categories}

    In the syntax, substitution is defined recursively and therefore satisfies definitional equalities.
    These equalities are not even expressible in a display map category because substitution is only specified up to isomorphism, and \(=\)-types are only specified up to equivalence.
    So, to model type theory strictly, we need specified choices: a \emph{structured} or \emph{cloven} display map category is equipped with a subclass of the display maps that generates the display maps by repletion (the \emph{strict} display maps), and choices for re-indexing strict display maps to strict display maps.
    We call it a \emph{strict} or \emph{split} display map category if these choices respect identity and composition: \(A[1_\Gamma]=A\) and \(A[\tau\sigma]=A[\tau][\sigma]\).
    A strict rooted display map category also has choices to write every \(\Gamma\to1\) as a composition of strict display maps and isomorphisms.
    Given specified choices for a type former, we can consider a range of conditions \cite{MR3372323}: \begin{itemize}
        \item \makesize{\emph{weakly}}{\emph{strictly}} \emph{stable}: preserved by re-indexing up to \makesize{equivalence}{isomorphism} over the context.
        \item \makesize{\emph{pseudo}}{\emph{strictly}}-\emph{stable}: preserved by re-indexing up to isomorphism over the context.
        \item \makesize{\emph{strictly}}{\emph{strictly}} \emph{stable}: preserved by re-indexing.
    \end{itemize}
    For example, specified choices for =-types consist of: for every strict type \(A\) a strict type \(\Id_A\) and a term \(\refl_A\), together with for every strict type \(C\) and term \(d\), terms \(\ind^=_d\) and \(\beta^=_d\).
    These choices are \emph{strictly stable} if for any \(\sigma\) we have the following equalities: \begin{align*}
        \Id_A[\sigma^{\liltriangle\liltriangle}]&=\Id_{A[\sigma]}, & \ind^=_d[\sigma^{\liltriangle\liltriangle}]&=\ind^=_{d[\sigma^\liltriangle]}, &
        \refl_A[\sigma^\liltriangle]&=\refl_{A[\sigma]}, & \beta^=_d[\sigma^\liltriangle]&=\beta^=_{d[\sigma^\liltriangle]}.
    \end{align*}
    A \emph{weak model} of a type theory is a display map category with weakly stable type formers, and a \emph{strict model} is a strict display map category with strictly stable type formers.

    \section{Path categories}\label{sec:pathcat}

    We start with the following notion from abstract homotopy theory:
    \begin{definition}[Clan \cite{joyal_notes_2017}]
        A \emph{clan} is a category \(\cal C\) with a subclass of maps \(\cal C^{\sf{fib}}\subseteq\cal C^\to\) called the \emph{fibrations (\(\fib\))} that is replete, closed under identity, composition, and pullback along any map.
        In addition, there exists a terminal object \(1\) and every map \(A\to 1\) is a fibration.
    \end{definition}
    We have a forgetful functor from rooted display map categories to clans (only remembering the fibrations, that is, the compositions of display maps), and conversely every clan gives a rooted display map category in a cofree way: the entire class of fibrations can be taken as the display maps.
    To emphasize the topological view, we start writing the pullback of \(\Omega\fib\Gamma\) and \(\sigma:\Delta\to\Gamma\) as \(\Delta\times_\Gamma\Omega\) instead of \(\Delta.\Omega[\sigma]\).
    We extend clans with equivalences:
    \begin{definition}[Path category \cite{MR3795638}]\label{pathcat}
        A \emph{path category} is a clan with an additional subclass \(\cal C^{\sf{eq}}\subseteq\cal C^\to\) of \emph{(weak) equivalences (\(\tto\))}.
        If a map is both a fibration and an equivalence, we call it a \emph{trivial fibration (\(\tfib\))}.
        In addition, we require:
        \begin{enumerate}
            \item
            isomorphisms are weak equivalences,
            \item
            equivalences satisfy 2-out-of-6: if for \(\smash{A\stackrel f\to B\stackrel g\to C\stackrel h\to D}\) we have that both \(gf\) and \(hg\) are equivalences, then so are \(f\), \(g\), \(h\), and \(hgf\),
            \item
            the pullback of a trivial fibration along any map is also a trivial fibration,
            \item
            every trivial fibration has a section,
            \item
            every object \(A\) has a path object: an object \(PA\) together with a fibration \((s,t):PA\fib A\times A\) and an equivalence \(r:A\tto PA\) with \((s,t)r=(1,1)\).
        \end{enumerate}
    \end{definition}
    This is a strengthening of Brown's notion of a category of fibrant objects \cite{Brown1973} in two ways: (a) we have 2-out-of-6 instead of the weaker 2-out-of-3 (if for \(\smash{A\to B\to C}\) two of the three maps are equivalences, then so is the third), which is equivalent to saturation \cite{RadulescuBanu2009}, and (b) every trivial fibration has a section, which is equivalent to every object being cofibrant. 

    A trivial example of a path category is a category with finite limits where every morphism is a fibration, and only isomorphisms are equivalences.
    More interesting examples arise by taking the fibrant-cofibrant objects of a model category where every object is cofibrant. 
    In addition, the effective topos can be seen as the homotopy category of a path category \cite{Van_Den_Berg_2020}.
    
    Four main results for path categories are: factorization, slicing, lifting, and saturation, which we will show here; for proofs, see \cite{MR3795638}.
    First, notice that finite products exist and that the projections are pullbacks of fibrations and therefore fibrations.
    In addition, we see that the source and target maps \(s,t:PA\to A\) are always trivial fibrations: equivalences by 2-out-of-3 on \(sr=1\) and \(tr=1\), and fibrations as compositions of fibrations because \(s=\pi_0(s,t)\) and \(t=\pi_1(s,t)\).
    Now we can state:
    \begin{proposition}[Factorization]\label{prop:factorization}
        We can factor any map \(f:B\to A\) as an equivalence \(w_f:B\tto Lf\) (a section of a trivial fibration) followed by a fibration \(p_f:Lf\fib A\).
    \end{proposition}
    Using generalized elements, we think of the \emph{mapping path space} \(Lf\coloneqq B\mathbin{_f\times_s}PA\) as the space of pairs \((b,\alpha)\) with \(b\) in \(B\) and \(\alpha\) a path in \(A\) with source \(fb\): \[\hfill
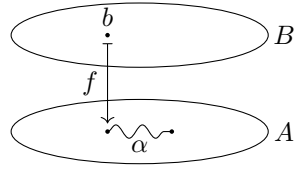
\begin{tikzpicture}[scale = 0.425]
        \fill (-1,3) node[above] {\(b\)} circle[radius = 2pt];
        \fill (-1,0) circle[radius = 2pt];
        \fill (1,0) circle[radius = 2pt];
        \draw[decorate, decoration=snake] (-1,0) -- node[below] {\(\alpha\)} (1,0);
        \draw[|->] (-1,2.75) -- node[left] {\(f\)} (-1,0.25);
        \draw (0,3) ellipse (4 and 1);
        \node at (4.5,3) {\(B\)};
        \draw (0,0) ellipse (4 and 1);
        \node at (4.5,0) {\(A\)};
    \end{tikzpicture}\hfill\]
    Then \(w_f\coloneqq(1_B,rf)\) maps \(b\) in \(B\) to the pair consisting of \(b\) and the trivial path on \(fb\), and \(p_f\coloneqq t\pi_1\) maps \((b,\alpha)\) to the target of \(\alpha\).
    \begin{proof}
        The projection \(\pi_0:B\mathbin{_f\times_s}PA\to B\) is the pullback of \(s\) and therefore a trivial fibration.
        So, with 2-out-of-3 on \(\pi_0(1_B,rf)=1_B\) we see that \(w_f\coloneqq(1_B,rf)\) is an equivalence.
        In addition, we see that the following square is pullback: \[\hfill\begin{tikzcd}[column sep = large, row sep = small]
            B\mathbin{_f\times_s}PA \ar[r, "\pi_1"] \ar[d, "{(\pi_0,t\pi_1)}"'] \ar[rd, "\lrcorner"{anchor=center, pos=0.125}, draw=none] & PA \ar[d, "{(s,t)}"{near start}, fib] \\
            B\times A \ar[r, "{(f\pi_0,\pi_1)}"'] & A\times A
        \end{tikzcd}\hfill\]
        So, \((\pi_0,t\pi_1)\) is a fibration, and therefore \(p_f\coloneqq t\pi_1=\pi_1(\pi_0,t\pi_1)\) is as well.
    \end{proof}
    \begin{definition}[Slice category of fibrations]
        If \(\cal C\) is a path category then for any object \(\Gamma\) in \(\cal C\) we define \(\cal C/^{\sf{fib}}\Gamma\) as the full subcategory of \(\cal C/\Gamma\) whose objects are fibrations to \(\Gamma\).
        This is a path category when notions of fibration and equivalence are copied from \(\cal C\).
        The path object \(A\tto P_\Gamma A\fib A\times_\Gamma\times A\) of a fibration \(A\fib\Gamma\) is a factorisation of \(\delta_A:A\to A\times_\Gamma A\) in \(\cal C\).
    \end{definition}
    \begin{definition}[Homotopy]
        We write \(f\simeq g\) for \(f,g:A\to B\) if there exists an \emph{homotopy}: an \(h:A\to PB\) such that \((s,t)h=(f,g)\).
        A \emph{homotopy equivalence} is a map \(f:A\to B\) such that there exists a \(g:B\to A\) with \(gf\simeq 1\) and \(fg\simeq 1\).
        More generally, if \(B\fib\Gamma\) then we write \(f\simeq_\Gamma g\) if there exists a \emph{fibrewise homotopy}: an \(h:A\to P_\Gamma B\) such that \((s,t)h=(f,g)\).
    \end{definition}
    \begin{theorem}[Saturation]\label{the:saturation}
        The homotopy equivalences are precisely the weak equivalences.
    \end{theorem}
    \begin{theorem}[Lifting]\label{the:lifting}
        Suppose that we have the commutative outer square: \[\hfill\begin{tikzcd}[column sep = large, row sep=scriptsize]
            B \ar[r, "f"] \ar[d, "w"', equiv] \ar[rd, phantom, "\simeq_\Gamma"{near start}] & A \ar[d, "p"{near start}, fib] \\
            \Delta \ar[r, "\sigma"'] \ar[ru, "l"'{inner sep=0pt}, dashed] & \Gamma.
        \end{tikzcd}\hfill\]
        Then there exists a map \(l:\Delta\to A\) such that the lower triangle commutes and the upper triangle commutes up to fibrewise homotopy: we have \(pl=\sigma\) and \(lw\simeq_\Gamma f\).
        Moreover, such a lifting \(l\) is unique up to fibrewise homotopy.
    \end{theorem}
    \begin{corollary}
        Let $f : B \to A$ be any morphism in a path category and let $PB$ and $PA$ be path objects. Then there is a morphism $P f : P B \to P A$ commuting with the morphisms $(s_B,t_B)$ and $(s_A,t_A)$ and such that $(P f)r_B \simeq_{A\times A} r_A f$. In particular, homotopy is independent of the choice of path objects (consider $f=1_A$).
    \end{corollary}
    Although they are 1-categories, path categories already contain higher structure.
    In particular, they are enriched in \(\infty\)-groupoids \cite{paauw2021path}: a map between morphisms \(f,f':A\to B\) is an homotopy, that is, a morphism \(h:A\to PB\), and so on.
    
    We have a 2-category \(\sf{PathCat}\) of path categories, functors preserving the structure (fibrations, trivial fibrations, pullbacks, and the terminal object), and natural transformations.
    Such functors also preserve equivalences and therefore path objects by Brown's lemma \cite{Brown1973}.

    We call path categories \emph{structured} if there is a choice of fibrations (the \emph{strict} fibrations) that generates the fibrations by repletion, and choices of path fibrations, pullbacks of strict fibrations to strict fibrations, and terminal object.
    We call it \emph{strict} if these choices strictly respect identity, composition, and path objects.
    
    \section{From path categories to weak models and back}\label{sec:iso}
    
    We anticipated that path categories would provide a natural semantics for axiomatic \(=\)-types.
    In this section, we establish the rigorous correspondence by rephrasing a path category as a display map category with weakly-stable axiomatic \(=\)-types and three other properties: a root, (pseudo-stable) extensional \(\bb1\)-types, and (pseudo-stable) extensional \(\Sigma\)-types.
    In \autoref{sec:disppathcat}, we introduce a version of path categories that does not necessarily model extensional \(\bb1\)-types and $\Sigma$-types, and hence provides semantics for a minimal axiomatic type theory with only axiomatic \(=\)-types.
    We will also see how we can extend this with other axiomatic type formers.
    Let us start from these last properties:

    \begin{proposition}\label{prop:clanasdisp}
        A display map category has: \begin{itemize}
            \item extensional \makesize[r]{\(\bb1\)}{\(\Sigma\)}-types iff all identity maps are display maps;
            \item extensional \(\Sigma\)-types iff display maps are closed under composition.
        \end{itemize}
        So, a display map category is equal to a clan (as categories with a specified subclass of maps) iff it is rooted and has extensional \(\bb1\)- and \(\Sigma\)-types.
        In addition, extensional \(\bb1\)- and \(\Sigma\)-types are always pseudo-stable.
    \end{proposition}
    \begin{proof}
        If we have extensional \(\bb1\)-types, then for any \(\Gamma\), the display map \(\Gamma.\bb1\disp\Gamma\) is an isomorphism by the introduction and \(\beta,\eta\)-rules.
        As any isomorphism of \(\cal C\) is isomorphic to the identity in \(\cal C^\to\), the latter is a display map as well.
        Similarly, if we have extensional \(\Sigma\)-types then for any composable display maps \(\Gamma.A.B\disp\Gamma.A\disp\Gamma\), we have a display map \(\Gamma.\Sigma_AB\disp\Gamma\) that is isomorphic in \(\cal C^\to\) to the composition, hence the latter is a display map as well. Vice versa, define \(\bb1_\Gamma\) as \(1_\Gamma:\Gamma\to\Gamma\) and \(\Sigma_AB\) as \(\Gamma.A.B\disp\Gamma.A\disp\Gamma\).

        Pseudo-stability follow from the fact that isomorphisms are preserved by re-indexing.
    \end{proof}
    To show that path categories also have axiomatic \(=\)-types, we adapt \cite[Proposition 4.5]{MR3828037}:
    
    \begin{restatable}{proposition}{pathtodisp}\label{prop:pathtodisp}
        A path category has weakly stable based axiomatic \(=\)-types.
        Therefore --- by considering every fibration to be a display map --- a path category is a rooted display map category with weakly stable axiomatic \(=\)-types, and extensional \(\bb1\)- and \(\Sigma\)-types.
    \end{restatable}

    \begin{proof}[Proof Sketch.]
        Formation and introduction for a type $A$ in context $\Gamma$ are provided by a path object for the fibration $A\fib\Gamma$. Elimination and the $\beta$-axiom follow by the lifting theorem (\autoref{the:lifting}).
        Weak stability follows from the fact that the re-indexing of a path object is a path object itself.
        See \autoref{app:iso} for a complete proof.
    \end{proof}
    For the converse, we first translate different notions of axiomatic =-types:
    
    \begin{restatable}{proposition}{identityvariations}\label{prop:identityvariations}
        For a display map category, the existence of the following weakly stable axiomatic identity types is equivalent: \[\hfill
            \text{unparametrized based }\liff\text{ parametrized based }\liff\text{ parametrized unbased},
        \hfill\] whereas the existence of weakly stable unparametrized unbased axiomatic $=$-types is weaker: transport is not definable with unbased path induction \emph{(and no \(\Pi\)-types)}.
    \end{restatable}

    \begin{proof}
        The proof of the first equivalence is obtained by adapting Bocquet's argument \cite[Section 3]{Bocquet2020} to non-strict models.
        Regarding the second equivalence, a parametrized unbased axiomatic $=$-type for a type $A$ in a given context $\Gamma$ can be defined as a specific re-indexing of a parametrized based axiomatic $=$-type for \textit{some} $A^\liltriangle$ over $\Gamma.A$ --- namely, along the corresponding morphism $\delta_A^\liltriangle$.
        
        The strategy to prove that transport is not admissible in a theory with only the Martin-L\"of elimination and computation rules and no \(\Pi\)-types --- and hence that weakly stable unparametrized unbased axiomatic $=$-types are weaker than the other notions --- adapts an argument by Bocquet \cite[page 15]{Bocquet2020}. We also refer the reader to \cite{MR2469279}.
        
        We refer the reader to \autoref{app:iso} for a complete proof.
    \end{proof}
  
    \begin{restatable}{proposition}{disptopath}\label{prop:disptopath}
        A rooted display map with weakly stable based axiomatic $=$-types and extensional \(\bb1\)- and \(\Sigma\)-types can be given the structure of a path category where the weak equivalences are defined as the homotopy equivalences.
    \end{restatable}
    \begin{proof}[Proof Sketch.]
        By Proposition \ref{prop:identityvariations} and Propositional \ref{prop:clanasdisp} --- by looking at its display map class as a class of fibrations --- we infer that \(\cal C\) is a clan and is endowed with weakly stable \textit{parametrized unbased} axiomatic $=$-types. Now, one can observe that weak stability implies the existence of arrows: \begin{equation}\label{equ:logicalequivalence}    \Id_\Omega[\sigma^{\liltriangle\liltriangle}]\to\Id_{\Omega[\sigma]}\qquad\text{and}\qquad\Id_{\Omega[\sigma]}\to\Id_\Omega[\sigma^{\liltriangle\liltriangle}] \tag{logical equivalence}
        \end{equation} over the canonical isomorphism $\Omega[\sigma]\times_\Delta\Omega[\sigma] \iso (\Omega \times_\Gamma\Omega)[\sigma^\liltriangle]$, for every fibration $\Omega\fib\Gamma$ and every substitution $\Delta\xrightarrow{\sigma}\Gamma$. In other words, the equivalence relations represented by $\Id_\Omega[\sigma^{\liltriangle\liltriangle}]$ and $\Id_{\Omega[\sigma]}$ coincide. We provide full verification of this fact in \autoref{app:iso}.
        Van den Berg \cite[Theorem 5.16]{MR3828037} proves that these data on $\mathcal{C}$ are in fact equivalent to the existence and the choice of a class of weak equivalences on $\mathcal{C}$ --- in fact provided by its homotopy equivalences --- making the clan $\mathcal{C}$ into a path category, concluding our argument.
    \end{proof}

    The correspondence delineated by \autoref{prop:pathtodisp} and \autoref{prop:disptopath} extends to an equality of the 2-categories; for which we provide details in \autoref{app:iso}:
    \begin{restatable}{theorem}{pathasdisp}\label{the:pathasdisp}
        There exists an isomorphism of 2-categories: $\sf{PathCat}\cong\sf{DispCat}^{\sf{root}}_{=_{\sf{axi}},\bb1_{\sf{ext}},\Sigma_{\sf{ext}}}$.
    \end{restatable}

    \section{Display map path categories}\label{sec:disppathcat}

    To define a more fine-grained semantics that does not necessarily model extensional \(\bb1\)-types and extensional \(\Sigma\)-types, we will add some additional structure to path categories: we exchange the clan for a display map category to distinguish display maps and fibrations.
    
    \begin{definition}[Display map path category]
        A \emph{display map path category} is a display map category \(\cal C\) together with a class \(\cal C^{\sf{eq}}\subseteq\cal C^\to\) of \emph{(weak) equivalences (\(\tto\))} such that: \begin{description}
            \item[\emph{1.}]
            isomorphisms are equivalences,
            \item[\emph{2.}]
            equivalences satisfy 2-out-of-6: if for \(\smash{A\overset f\to B\overset g\to C\overset h\to D}\) we have that both \(gf\) and \(hg\) are equivalences, then so are \(f\), \(g\), \(h\), and \(hgf\),
            \item[\emph{3.}]
            the pullback of a trivial fibration along any map is also a trivial fibration,
            \item[\emph{4.}]
            every trivial fibration has a section,
            \item[\emph{5.}]
            every display map \(A\disp\Gamma\) has a \emph{path display map}: \\
            a display map \((s,t):P_\Gamma A\disp A\times_\Gamma A\) and equivalence \(r:A\tto P_\Gamma A\) with \((s,t)r=(1,1)\).
            \item[\emph{\graybf{PF.}}]
            every fibration \(\Delta\fib\Gamma\) has a \emph{path fibration}: \\
            a fibration \((s,t):P_\Gamma\Delta\fib \Delta\times_\Gamma\Delta\) and equivalence \(r:\Delta\tto P_\Gamma\Delta\) with \((s,t)r=(1,1)\).
        \end{description}
    \end{definition}
    We call a display map path category \emph{rooted} if the underlying display map category is.
    Every rooted display map path category is in particular a path category, and every path category gives a cofree rooted display map path category by considering every fibrations to be a display map.
    Now a natural question is: are path display maps \(P_\Gamma A\disp A\times_\Gamma A\) already enough to define path fibrations \(P_\Gamma\Delta\fib\Delta\times_\Gamma\Delta\)?
    This is the case in the syntax: with \(=\)-types (and transport) we can define =-telescopes.
    For a telescope \(\vec x:\vec A\) of \(\Gamma\) we define the =-telescope \(\vec\chi:\vec\chi_*\vec x=_{\vec A}\vec x'\) of \(\Gamma,\vec x:\vec A,\vec x':\vec A\) where: \begin{alignat*}{4}
        \chi_i&:(\chi_{i-1})_*\,\cdots\,(\chi_0)_*\,x_i=_{A_i[x_0',\dots,x_{i-1}'/x_0,\dots,x_{i-1}]}x_i'.
    \end{alignat*}
    The same is true in general.
    However, we do need to assume a notion of transport:\begin{theorem}\label{the:equivalent_axioms}
        In the presence of axioms \emph{\graybf{1--5}} and a root, the following are equivalent: \begin{description}
            \item[\emph{PO.}]
            \emph{(path objects)} every object \(\Gamma\) has a path object \(P\Gamma\).
            \item[\emph{PF.}]
            \emph{(path fibrations)} every fibration \(\Delta\fib\Gamma\) has a path fibration \(P_\Gamma\Delta\).
            \item[\emph{F.}]
            \emph{(factorisation)} every map \(f:B\to A\) has a factorisation \(B\tto Lf\fib A\).
            \item[\emph{T.}]
            \emph{(transport)} for any \(p:A\disp\Gamma\) and path object \(P\Gamma\) of \(\Gamma\) there exists a map \(\tau:Lp\to A\) for \(Lp\coloneqq A\mathbin{_p\times_s}P\Gamma\) such that \(p\tau=t\pi_1\) and \(\tau(1,rp)\simeq_\Gamma1\).
            \item[\emph{L.}]
            \emph{(lifting)} if we have a commutative outer square \[\hfill\begin{tikzcd}[column sep = large, row sep=1.5em]
                B \ar[r, "f"] \ar[d, "w"', equiv] \ar[rd, phantom, "\simeq_\Gamma"{near start}] & A \ar[d, "p"{near start}, fib] \\
                \Delta \ar[r, "\sigma"'] \ar[ru, "l"'{inner sep=0pt}, dashed] & \Gamma\mathrlap,
            \end{tikzcd}\hfill\] then there exists a map \(l:\Delta\to A\) such that the lower triangle commutes and the upper triangle commutes up to fibrewise homotopy, that is, we have \(pl=\sigma\) and \(lw\simeq_\Gamma f\).
            Moreover, such a lifting is unique up to fibrewise homotopy.
        \end{description}
    \end{theorem}
    \begin{proof} We have the following implications: \begin{itemize}
            \item (\(\text{\graybf{PO}}\to\text{\graybf{PF}}\to\text{\graybf{F}}\to\text{\graybf{PO}}\)) follows by \autoref{prop:factorization} and the fact that the path object \(P\Gamma\) is a factorisation of the diagonal \(\delta_\Gamma:\Gamma\to\Gamma\times\Gamma\).
            \item(\(\text{\graybf{PO}}\to\text{\graybf{L}}\)) by the lifting theorem for path categories (\autoref{the:lifting}).
            \item(\(\text{\graybf{L}}\to\text{\graybf{T}}\)) because a transport map is a lift for the following square: \[\hfill\begin{tikzcd}[column sep = large, row sep = 1.5em]
                B \ar[r, "f"] \ar[d, "w_p"', equiv] \ar[rd, phantom, "\simeq_\Gamma"{near start}] & A \ar[d, "p"{near start}, fib] \\
                Lp \ar[r, "p_p"', fib] \ar[ru, "\tau"'{inner sep=0pt}, dashed] & \Gamma,
            \end{tikzcd}\hfill\]
            \item(\(\text{\graybf{T}}\to\text{\graybf{PO}}\)) follows with induction on the number of display maps in the fibration \(\Gamma\fib1\).
            In the base case we take the path object to be \(1\).
            In the successor case we will combine a path object \(PA\disp^n A\times A\) for \(A\disp^n1\) with a path object \(P_AB\disp B\times_AB\) for \(p:B\disp A\) into a path object \(PB\disp^{n+1}B\times B\) for \(B\disp^{n+1}1\).
            Take \(Lp\coloneqq B\mathbin{_p\times_s}PA\) and let \(\tau:Lp\to A\) be a transport map.
            The intuition for \(PB\) will be that it consists of quadruples \((b,b',\alpha,\beta)\) as in the following picture: \[\hfill
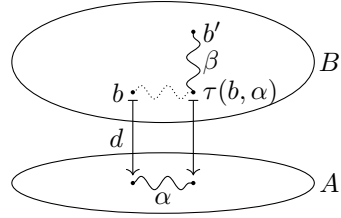
\begin{tikzpicture}[scale = 0.4]
                \fill (-1,3) node[left] {\(b\)} circle[radius = 2pt];
                \fill (1,3) node[right] {\(\tau(b,\alpha)\)} circle[radius = 2pt];
                \fill (-1,0) circle[radius = 2pt];
                \fill (1,0) circle[radius = 2pt];
                \fill (1,5) node[right] {\(b'\)} circle[radius = 2pt];
                \draw[densely dotted, decorate, decoration=snake] (-1,3) -- (1,3);
                \draw[decorate, decoration=snake] (-1,0) -- node[below] {\(\alpha\)} (1,0);
                \draw[decorate, decoration=snake] (1,3) -- node[right] {\(\beta\)} (1,5);
                \draw[|->] (-1,2.75) -- node[left] {\(d\)} (-1,0.25);
                \draw[|->] (1,2.75) -- (1,0.25);
                \draw (0,4) ellipse (5 and 2);
                \node at (5.5,4) {\(B\)};
                \draw (0,0) ellipse (5 and 1);
                \node at (5.5,0) {\(A\)};
            \end{tikzpicture}\hfill\]
            More precisely, we define \(PB\) using the two pullbacks: \[\hfill\begin{tikzcd}[column sep = 7em, row sep = small]
                P_AB \ar[d, disp, "{(s,t)}"'] & PB \ar[l, "\pi_1"'] \ar[d, disp, "\pi_0"] \ar[ld, "\llcorner"{anchor=center, pos=0.125}, draw=none] \\
                B\times_AB & (B\times B)\times_{A\times A}PA \ar[l, "{(\tau(\pi_0\pi_0,\pi_1),\pi_1\pi_0)}"] \ar[r, "\pi_1"] \ar[d, disp, "n"{pos=1}, "\pi_0"'] \ar[rd, "\lrcorner"{anchor=center, pos=0.125}, draw=none] & PA \ar[d, disp, "n"{pos=1}, "{(s,t)}"] \\
                & B\times B \ar[r, "{p\times p}"'] & A\times A
            \end{tikzcd}\hfill\] and we take \(r_B\coloneqq(((1,1),rp),r)\) and \((s_B,t_B)\coloneqq\pi_0\pi_0\).
            To see that \(r_B\) is an equivalence, we note that there is a more concise isomorphic definition of \(PB\) as the mapping path space \(L_A\tau\) of \(\tau:Ld\to B\) in \(\cal C_A\): \[
                L_A\tau\coloneqq Ld\mathbin{_\tau\times_s}P_AB=(B\mathbin{_p\times_s}PA)\mathbin{_\tau\times_s}P_AB.
            \]
            We see that the following diagram commutes: \[\hfill\begin{tikzcd}[row sep = 1.8em]
                B \ar[d, "w_d"', equiv] \ar[r,"r_B"] & PB \ar[d, "{(\pi_0\times 1_{P_AB})\times 1_{PA}}"{inner sep=7pt}, iso'] \\
                Ld & L_A\tau\mathrlap, \ar[l, "\pi_0", equiv]
            \end{tikzcd}\hfill\] so, with 2-out-of-3, we see that \(r_B\) is an equivalence.\qedhere
        \end{itemize}
    \end{proof}
    So, the main results for path categories turn out to be equivalent to the existence of path fibrations.
    Another observation is that the slice \(\cal C/\Gamma\) of a (rooted) display path category is generally only a display map path category, whereas \(\cal C/^{\sf{fib}}\Gamma\) is a rooted display map path category.
    So, both the rooted and the unrooted version are important regardless of the starting category, which we exploit in the following:
    \begin{theorem}
        A display map path category has weakly stable =-types, and every display map category with weakly stable =-types can be given a class of weak equivalences that turns it into a display map path category: the fibrewise homotopy equivalences.
        This extends to an isomorphism of the 2-categories: \(\sf{DispPathCat}^{\sf{root}}\cong\sf{DispCat}_{=_{\sf{axi}}}^{\sf{root}}\).
    \end{theorem}
    \begin{proof}
        If we have a display map path category, then for every \(A\disp\Gamma\) we have a path object \(P_\Gamma A\disp A\times_\Gamma A\) in the rooted display path category \(\cal C/^{\sf{fib}}\Gamma\).
        This is in particular a path category, so by \autoref{prop:pathtodisp} it is a =-type in \(\cal C/^{\sf{fib}}\Gamma\), which implies that it is a =-type in \(\cal C\).
        The =-types are weakly stable because we see that display path objects are preserved by pullback using (the proof of) Brown's lemma \cite[page 428]{Brown1973}.

        Conversely, if we have a display map category with weakly stable axiomatic \(=\)-types, then we say that a map \(f:A\to B\) is an equivalence if there exists a \(\Gamma\) such that \(A\fib\Gamma\) and \(B\fib\Gamma\) and \(f\) is a homotopy equivalence over \(\Gamma\).
        We see that this satisfies axioms \graybf{1--4} so the main difficulty is showing that every trivial fibration has a section and that every fibration has a path fibration.
        For a (trivial) fibration \(f:\Delta\tfib\Omega\) there exists a \(\Gamma\) such that \(f\) is an equivalence over \(\Gamma\).
        So, we consider \(\cal C/^{\sf{fib}}\Gamma\), which has a root, and Van den Berg \cite[Section 6]{MR3828037} shows that parametrized unbased axiomatic \(=\)-types are make it a path category whose equivalences are again the homotopy equivalences.
        So, the remaining properties follow.

        If there is a root, then we know that saturation holds (\autoref{the:saturation}), and the fibrewise homotopy equivalences become the homotopy equivalences, so this is the only option for the weak equivalences.
        In a similar way as for \autoref{the:pathasdisp} we see that the 1-cells agree.
    \end{proof}

    \begin{proposition}
        A (display map) path category: \begin{itemize}
            \item has weakly stable axiomatic \(\bb1\)-types iff the identity maps are homotopic to display maps;
            \item has weakly stable axiomatic \(\Sigma\)-types iff the display maps are closed under composition up to homotopy;
            \item has weakly stable axiomatic \(\Pi\)-types iff for every $B\disp A\disp\Gamma$ there exist $\Pi_AB\disp\Gamma$ and \(\sf{app}_{A,B}:(\Pi_AB)\times_\Gamma A\to B\) over \(\Gamma\) such that for every \(\sigma:\Delta\to\Gamma\) the map \[
                (\mathcal{C}/\Gamma)(\Delta,\Pi_AB)\xrightarrow{-\times_\Gamma A}(\cal C/A)(\Delta\times_\Gamma A,(\Pi_AB)\times_\Gamma A)\xrightarrow{\sf{app}\circ -}(\mathcal{C}/A)(\Delta\times_\Gamma A,B),
            \] is essentially surjective.
            Here we view the domain and codomain as groupoids, where the maps are given by pointwise homotopies up to higher homotopy.
            \(\Pi_AB\) satisfies the function extensionality axiom iff the map is full, which implies that it is an equivalence.
        \end{itemize}
    \end{proposition}
    \begin{proof}
        For the \(\bb1\)- and \(\Sigma\)-types, the introduction and \(\beta\)-axiom give equivalences \(\Gamma.\bb1\tdisp\Gamma\) and \(\Gamma.\Sigma_A.B\tto\Gamma.A.B\) over \(\Gamma\).
        So, the identity and composition are homotopic to display maps.
        Conversely, we can define \(\bb1_\Gamma\) as the display map homotopic to \(1_\Gamma:\Gamma\to\Gamma\) and \(\Sigma_AB\) as the display map homotopic to \(B\disp A\disp\Gamma\).
        They are always weakly stable because display maps and equivalences are preserved by re-indexing.

        For the \(\Pi\)-types, these correspond to the weak/strong homotopy exponentials of Den Besten \cite{den_besten_homotopy_2020}.
        Essentially surjective gives a form of the introduction and \(\beta\)-axiom with weak stability build in: if we have a substitution \(\sigma:\Delta\to\Gamma\) and a term \(b:\Delta.A[\sigma]\to\Gamma.B\) over \(\sigma p_{A[\sigma]}\) then we get a term \(\lambda b:\Delta\to\Gamma.\Pi_AB\) over \(\sigma\) such that \(\sf{app}(\lambda b,1_{\Gamma.A})\) is homotopic to \(b\).
        That the map being full implies that it is an equivalence, and that this is equivalent the axiom of function extensionality is shown by Den Besten \cite[Proposition 5.4]{den_besten_homotopy_2020}.
    \end{proof}
    
    \section{Strictification and coherence}\label{sec:strictify}

    In this last section, we use our isomorphisms to prove a coherence result: a biequivalence between a certain class of path categories and the class of (strict) models of a specific theory.
    First, we recall the LF condition that Lumsdaine and Warren \cite{MR3372323} use to strictify:

    \begin{definition}[Logical framework (LF)]
        A display map category satisfies \emph{LF} if it has finite products and if, whenever \(A\to\Gamma\) is a display map and \(B\to A\) is either a display map or a product projection, the categorical dependent exponent \(\Pi_AB\) exists: there exist $\Pi_AB\to\Gamma$ and \(\sf{app}_{A,B}:(\Pi_AB)\times_\Gamma A\to B\) over \(\Gamma\) such that for every \(\sigma:\Delta\to\Gamma\) the map: \[
            (\mathcal{C}/\Gamma)(\Delta,\Pi_AB)\xrightarrow{-\times_\Gamma A}(\cal C/A)(\Delta\times_\Gamma A,(\Pi_AB)\times_\Gamma A)\xrightarrow{\sf{app}\circ -}(\mathcal{C}/A)(\Delta\times_\Gamma A,B)\tag{LF}\label{def:LF}
        \] is an isomorphism.
    \end{definition}
    This condition is very close to the existence of \(\Pi\)-types for (display map) path categories but distinct in two important ways: we do not ask for \(\Pi_AB\to\Gamma\) to be a display map (so \(\Pi_AB\) is not a type), and we ask for an isomorphism instead of a full map or equivalence.
    As a consequence, it captures a situation that is better viewed as a logical framework \cite{harper_framework_1993} or a 2-level type theory \cite{MAIETTI2009319, altenkirch_extending_2016} where the ambient theory has strong \(\Pi\)-types regardless of the internal theory that we consider.

    \begin{restatable}[Path categories provide models]{theorem}{structuredpathtomodel}\label{the:structuredpathtomodel}
        Let $\mathcal{C}$ be a structured path category satisfying the LF condition.
        Then, seen as a structured display map category, it is equivalent in $\sf{DispCat}^{\sf{structured,LF,root}}_{=_{\sf{axi}},\bb1_{\sf{ext}},\Sigma_{\sf{ext}}}$ to an object \(\cal C_!\) of $\sf{DispCat}^{\sf{strict,LF,root}}_{=_{\sf{axi}},\bb1_{\sf{ext}},\Sigma_{\sf{ext}}}$ (a strict model).\footnote{Note that if we are willing to assume the axiom of choice in the metatheory, we do not need to assume that the display map category is structured.}
    \end{restatable}\noindent
    This follows by combining \autoref{the:pathasdisp} with the work of Lumsdaine and Warren \cite{MR3372323}, and Bocquet \cite{MR4481908}.
    As they work in different frameworks, we provide the details in \autoref{app:framework}.

    \begin{corollary}[Biequivalence of path categories and axiomatic Martin-L\"of type theories]\label{cor:biequivalencepathcategoriesandmodels}
        There is a biequivalence $\sf{PathCat}^{\sf{structured,LF}} \simeq \sf{DispCat}^{\sf{strict,LF,root}}_{=_{\sf{axi}},\bb1_{\sf{ext}},\Sigma_{\sf{ext}}}$.
    \end{corollary}
    \begin{proof}
        There is a choice of an equivalence $\eta_\mathcal{C}:\mathcal{C}\tto\mathcal{C}_!$ in $\sf{DispCat}^{\sf{structured,LF,root}}_{=_{\sf{axi}},\bb1_{\sf{ext}},\Sigma_{\sf{ext}}}$, natural in $\mathcal{C}$: the unit of the left-adjoint splitting.
        Now, if $\mathcal{D}$ is an object of $\sf{DispCat}^{\sf{strict,LF,root}}_{=_{\sf{axi}},\bb1_{\sf{ext}},\Sigma_{\sf{ext}}}$ then $\eta_{\mathcal{D}}$ --- where $\mathcal{D}$ is included into $\sf{DispCat}^{\sf{structured,LF,root}}_{=_{\sf{axi}},\bb1_{\sf{ext}},\Sigma_{\sf{ext}}}$ --- witnesses an equivalence $\mathcal{D}_! \simeq \mathcal{D}$ living in $\sf{DispCat}^{\sf{strict,LF,root}}_{=_{\sf{axi}},\bb1_{\sf{ext}},\Sigma_{\sf{ext}}}$.
        Therefore there is a biequivalence which, pre-composed by the restriction to the structured case of the biequivalence at \autoref{the:pathasdisp}, yields the statement.
    \end{proof}

    \begin{restatable}[Coherence for path categories]{corollary}{pathcatcoherence}\label{col:pathcatcoherence}
        Every structured LF path category is equivalent to a strict LF path category.     
    \end{restatable}\noindent
    See \autoref{app:framework} for the proof.
    These results work for general type formers, so we also get statements for display map path categories.
    The diagram summarises this: \[\begin{tikzcd}[column sep = -.8em]
        \phantom{\sf{Disp}}\sf{PathCat}_{(\Pi_{\sf{axi}})\phantom{(\bb1_\sf{axi}),(\Sigma_\sf{axi}),}}^{\sf{structured},\sf{LF}} \ar[d,bend left, "\text{cofree}"]\ar[phantom, d, "\dashv"] &
        \iso\sf{DispCat}_{=_{\sf{axi}},\bb1_{\sf{ext}},\Sigma_{\sf{ext}},(\Pi_{\sf{axi}})\phantom{()()}}^{\sf{structured},\sf{LF},\sf{root}} \ar[d,bend left, "\text{forget}"]\ar[phantom, d, "\dashv"] &
        \tto\sf{DispCat}_{=_{\sf{axi}},\bb1_{\sf{ext}},\Sigma_{\sf{ext}},(\Pi_{\sf{axi}})\phantom{()()}}^{\sf{strict},\sf{LF},\sf{root}} \\
        \sf{DispPathCat}_{(\bb1_{\sf{axi}}),(\Sigma_{\sf{axi}}),(\Pi_{\sf{axi}})}^{\sf{structured},\sf{LF},\sf{root}} \ar[u,bend left,"\text{forget}"] &
        \iso\sf{DispCat}_{=_{\sf{axi}},(\bb1_{\sf{axi}}),(\Sigma_{\sf{axi}}),(\Pi_{\sf{axi}})}^{\sf{structured},\sf{LF},\sf{root}} \ar[u,bend left,"\text{free}"] &
        \tto\sf{DispCat}_{=_{\sf{axi}},(\bb1_{\sf{axi}}),(\Sigma_{\sf{axi}}),(\Pi_{\sf{axi}})}^{\sf{strict},\sf{LF},\sf{root}}.
    \end{tikzcd}\]
    Note that the adjunctions, although labelled differently, are the same: depending on the point of view, the left adjoint is either forgetful (it forgets display maps while remembering the fibrations), or free (it takes the identity and composition closure of the display maps), while the right adjoint is either cofree (it takes every fibration to be a display map) or forgetful (it forgets the \(\bb1\) and \(\Sigma\)).

    \section{Conclusion}

    We have established a biequivalence between path categories and models of axiomatic $=$-types extended with extensional $\bb1-$ and $\Sigma$-types.
    This result provides a clean categorical formulation of the semantics of axiomatic =-types, allowing us to construct strict models directly from weak homotopy-theoretic structures.
    Furthermore, we introduced the novel notion of a display map path category.
    This more fine-grained structure distinguishes between types and telescopes, and offers a semantics for minimal axiomatic type theories that do not include extensional type formers.
    As we saw, this leaves room to add other axiomatic type formers, such as axiomatic $\bb1$-, \(\Sigma\)-, and \(\Pi\)-types, providing a modular semantics for axiomatic type theory.
    Our framework lays the groundwork for future investigations into the semantics of higher inductive types and computational interpretations of univalence.
    
    \bibliography{references}
    
    \appendix

    \section{Full syntax and semantics of axiomatic type theory}\label{app:syntaxsemantics}

    The structural rules of axiomatic type theory are: \[
        \prfstack[r]{ctx-empty,}
            {\epsilon\ctx} \qquad
        \prfstack[r,l]{ctx-extend,}{\(x\) fresh}
            {\Gamma\vdash A\type}
            {\Gamma,x:A\ctx} \qquad
        \prfstack[r]{var.}
            {\Gamma,x:A,\Omega\ctx}
            {\Gamma,x:A,\Omega\vdash x:A}
    \]
    These are modelled by any display map category with a designated empty context.
    In a rooted display map category this is the terminal object.

    \subsection{=-types}
    
    The rules for axiomatic based =-types are:\begin{align*}
        &\prfstack[r]{\(=\)F,}
            {\Gamma\vdash A\type \qquad \Gamma\vdash a:A \qquad \Gamma\vdash a':A}
            {\Gamma\vdash a=_Aa'\type}\, &
        &\prfstack[r]{\(=\)I,}
            {\Gamma\vdash a:A}
            {\Gamma\vdash\refl(a):a=_Aa} \\[2ex]
        &\prfstack[r]{\({=}\rm E_{\rm{based}}\),}
            {\Gamma,x':A,\chi:a=_Ax'\vdash C\type}
            {\Gamma\vdash d:C[a/x',\refl(a)/\chi]}
            {\Gamma\vdash\alpha:a=_Aa'}
            {\Gamma\vdash\sf{ind}^=_d(\alpha):C[a/x,a'/x',\alpha/\chi]}\, &
        &\prfstack[r]{\({=}\beta_{\rm{based}}^{\rm{ax}}\).}
            {\Gamma,x':A,\chi:a=_Ax'\vdash C\type}
            {\Gamma\vdash d:C[a/x',\refl(a)/\chi]}
            {\Gamma\vdash\beta^=_d:\sf{ind}^=_d(\refl(a))=d}
    \end{align*}
    This induces a notion in display map categories:
    \begin{definition}[\(=\)-types]
        We say that a display map category has \emph{axiomatic (based) $=$-types} if, for every type $A$ in context $\Gamma$, there exists:
        \begin{description}
            \item[F.] a type \(\Id_A\) in context $\Gamma.A.A^\liltriangle$,
            \item[I.] a term \(\refl_A\) of type \(\Id_A[\delta_A]\) in context \(\Gamma.A\), where \(\delta_A\) is the diagonal map \((1_{\Gamma.A},1_{\Gamma.A}):\Gamma.A\to\Gamma.A.A^\liltriangle\): \[\hfill\begin{tikzcd}[row sep = 1.5em, column sep = 5em]
            	\Gamma.A.\Id_A[\delta_A] \ar[d, disp]\ar[r,"\delta_A^\liltriangle"]\ar[dr, "\lrcorner"{anchor=center, pos=0.125}, draw=none] & \Gamma.A.A^\liltriangle.\Id_A \ar[d, disp] \\
                \Gamma.A \ar[r,"\delta_A"']\ar[u, bend left = 45, "\refl_A", dashed] & \Gamma.A.A^\liltriangle\mathrlap,
        	\end{tikzcd}\hfill\]
            \item[E.]
            for every term \(a\) of type \(A\), every type \(C\) in context \(\Gamma.A.\Id_A[a^\liltriangle]\), and every term \(d\) of type \(C[a,\refl[a]]\), a term \(\ind^=_d\) of \(C\): \[\hfill\begin{tikzcd}[column sep = 5em, row sep = 1.5em]
                \Gamma.C[a,\refl[a]]\ar[d, disp]\ar[r,"{(a,\refl[a])^\liltriangle}"]\ar[dr, "\lrcorner"{anchor=center, pos=0.125}, draw=none] & \Gamma.A.\Id_A[a^\liltriangle].C \ar[d, disp] \\
                \Gamma \ar[r,"{(a,\refl[a])}"']\ar[u, bend left = 45, "d"] & \Gamma.A.\Id_A[a^\liltriangle], \ar[u, bend right = 45, "{\ind_d^=}"', dashed]
        	\end{tikzcd}\hfill\]
            \item[\(\beta\).] for every \(a,C,d\) as before, \(\ind^=_d[a,\refl[a]]\) is propositionally equal to \(d\), that is, we have a term \(\beta^=_d\) of type \(\Id_{C[a,\refl[a]]}[\ind^=_d[a,\refl[a]],d]\).
        \end{description}
        Any such choice of data for a given type $A$ is also called an \emph{axiomatic $=$-type} for $A$. Additionally, we call an \(=\)-type \(\Id_A\) \emph{intensional} if it satisfies the \(\beta\)-rule: for every \(a,C,d\) as before we have that \(\beta^=_d=\refl[d]\), and therefore that $\ind^=_d[a,\refl[a]]=d$.
        We call the \(=\)-type \emph{extensional} if it also satisfies the \(\eta\)-rule: for every \(C\) and term \(c\) of \(C\) we have that $\ind^=_{c[a,\refl[a]]}=c$.
    \end{definition}
    In a strict display map category, we call a choice of =-types \emph{(strictly) stable} if for every type $A$ in context $\Gamma$ and every morphism $\sigma:\Delta\to\Gamma$, we have:\begin{align*}
        \Id_A[\sigma^{\liltriangle\liltriangle}]&=\makesize[r]\Id\refl_{A[\sigma]}, & \ind^=_d[\sigma^{\liltriangle\liltriangle}]&=\ind^=_{d[\sigma^\liltriangle]}, \\
        \refl_A[\sigma^\liltriangle]&=\refl_{A[\sigma]}, & \beta^=_d[\sigma^\liltriangle]&=\makesize[r]\beta\ind^=_{d[\sigma^\liltriangle]}.
    \end{align*}
    We say that the \(=\)-types are \emph{weakly stable} if for every type $A$ in context $\Gamma$ and every morphism $\sigma:\Delta\to\Gamma$ we can equip \(\Id_A[\sigma^{\liltriangle\liltriangle}]\) and \[\hfill\begin{tikzcd}[column sep = huge]
    	{\Delta.A[\sigma]} & {\Delta.A[\sigma].\Id_A[\delta_A][\sigma^\liltriangle]}
    	\arrow["{\refl_A[\sigma^\liltriangle]}", from=1-1, to=1-2]
    \end{tikzcd}\iso\Delta.A[\sigma].\Id_A[\sigma^{\liltriangle\liltriangle}][\delta_{A[\sigma]}],\hfill\] with the additional data required to form a \(=\)-type.

    \subsection{\texorpdfstring{\(\bb1\)}{1}-types}
    
    The rules for axiomatic \(\bb1\)-types are: \begin{align*}
        &\prfstack[r]{\(\bb1\)F,}
            {\Gamma\ctx}
            {\Gamma\vdash\bb1\type}\, &
        &\prfstack[r]{\(\bb1\)I,}
            {\Gamma\vdash 0_{\bb1}:\bb1} \\[2ex]
        &\prfstack[r]{\(\bb1\)E,}
            {\Gamma,w:\bb1\vdash C\type}
            {\Gamma\vdash d:C[0_{\bb1}/w]}
            {\Gamma\vdash i:\bb1}
            {\Gamma\vdash\sf{ind}^{\bb1}_d(i):C[i/w]}\, &
        &\prfstack[r]{\(\bb1\)\(\beta^{\sf{axi}}\).}
            {\Gamma,w:\bb1\vdash C\type}
            {\Gamma\vdash d:C[0_{\bb1}/w]}
            {\Gamma\vdash\beta^{\bb1}_d:\sf{ind}^{\bb1}_d(0_{\bb1})=d}
    \end{align*}
    This induces a notion in display map categories: \begin{definition}[\(\bb1\)-types]
        We say that a display map category has \emph{axiomatic $\bb1$-types} if, for every context $\Gamma$, there exists:
        \begin{description}
            \item[F.] a type $\bb1_\Gamma$ in context $\Gamma$,
            \item[I.] a term \(0_{\bb1}\) of type \(\bb1\): \[\hfill\begin{tikzcd}[column sep = large]
                \Gamma.\bb1 \ar[d, disp] \\
                \Gamma, \ar[u, dashed, bend right = 45, "0_{\bb1}"', pos = 0.4]
            \end{tikzcd}\hfill\]
            \item[E.] for every type \(C\) in context \(\Gamma.\bb1\) and every term \(d\) of type \(C[\bb1]\) a term \(\ind^{\bb1}_d\) of \(C\): \[\hfill\begin{tikzcd}[column sep = large]
                \Gamma.C[0_{\bb1}] \ar[d,disp]\ar[r,"\pair^\liltriangle"]\ar[dr, "\lrcorner"{anchor=center, pos=0.125}, draw=none] & \Gamma.\bb1.C \ar[d, disp] \\
                \Gamma \ar[u, bend left, "d"] \ar[r,"0_{\bb1}"'] & \Gamma.\bb1, \ar[u, bend right, "\ind^{\bb1}_d"', dashed, pos = 0.4]
            \end{tikzcd}\hfill\]
            \item[\(\beta\).] for every \(C\) and \(d\) as before, \(\ind^{\bb1}_d[0_{\bb1}]\) is propositionally equal to \(d\), that is, we have a term \(\beta^{\bb1}_d\) of type \(\Id_{C[0_{\bb1}]}[\ind^{\bb1}_d[0_{\bb1}],d]\).
        \end{description}
        We call a \(\bb1\)-type \(\bb1_\Gamma\) \emph{intensional} if it satisfies the \(\beta\)-rule: for every \(C\) and \(d\) we have that \(\beta^{\bb1}_d=\refl[d]\), and therefore that \(\ind^{\bb1}_d[0_{\bb1}]=d\).
        We call the \(\bb1\)-type \emph{extensional} if it also satisfies the \(\eta\)-rule: for every \(C\) and \(c\) of type \(C\) we have \(\ind^{\bb1}_{c[0_{\bb1}]}=c\).
    \end{definition}
    In a strict display map category, we call a choice of of \(\bb1\)-types \emph{(strictly) stable} if for every context $\Gamma$ and every morphism $\sigma:\Delta\to\Gamma$, we have: \begin{align*}
        (\bb1_\Gamma)[\sigma]&=\bb1_\Delta, &
        \ind^{\bb1}_d[\sigma^\liltriangle]&=\ind^{\bb1}_{d[\sigma^\liltriangle]}, \\
        0_{\bb1}[\sigma]&=0_{\bb1}, &
        \beta^{\bb1}_d[\sigma^\liltriangle]&=\makesize[r]\beta\ind^{\bb1}_{d[\sigma^\liltriangle]}.
    \end{align*}
    We say that the \(\bb1\)-types are \emph{weakly stable} if for every context $\Gamma$ and every morphism $\sigma:\Delta\to\Gamma$, we can equip \(\bb1_\Gamma[\sigma]\) and \(0_{\bb1}[\sigma]\) with the additional data required to form a \(\bb1\)-type.

    \subsection{\texorpdfstring{\(\Sigma\)}{Σ}-types}
    
    The rules for axiomatic \(\Sigma\)-types are: \begin{align*}
        &\prfstack[r]{\(\Sigma\)F,}
            {\Gamma,x:A\vdash B\type}
            {\Gamma\vdash\Sigma_{x:A}B\type}\, &
        &\prftree[r]{\(\Sigma\)I,}
            {\Gamma\vdash a:A}
            {\Gamma\vdash b:B[a/x]}
            {\Gamma\vdash\<a,b\>:\Sigma_{x:A}B} \\[2ex]
        &\prfstack[r]{\(\Sigma\)E,}
            {\Gamma,w:\Sigma_{x:A}B\vdash C\type}
            {\Gamma,x:A,y:B\vdash d:C[\<x,y\>/w]}
            {\Gamma\vdash p:\Sigma_{x:A}B}
            {\Gamma\vdash\sf{ind}^\Sigma_d(p):C[p/w]}\, &
        &\prfstack[r]{\(\Sigma\)\(\beta^{\sf{axi}}\).}
            {\Gamma,w:\Sigma_{x:A}B\vdash C\type}
            {\Gamma,x:A,y:B\vdash d:C[\<x,y\>/w]}
            {\Gamma\vdash a:A \qquad \Gamma\vdash b:B[a/x]}
            {\Gamma\vdash\beta^\Sigma_d(a,b):\sf{ind}^\Sigma_d(\<a,b\>)=d[a/x,b/y]}
    \end{align*}
    This induces a notion in display map categories: \begin{definition}[\(\Sigma\)-types]
        We say that a display map category has \emph{axiomatic $\Sigma$-types} if, for every type $A$ in context $\Gamma$ and type \(B\) in context \(\Gamma.A\), there exists:
        \begin{description}
            \item[F.] a type $\Sigma_AB$ in context $\Gamma$,
            \item[I.] a (generalised) term \(\pair_{A,B}\) of type \(\Sigma_AB\) for \(p_Ap_B:\Gamma.A.B\fib\Gamma\): \[\hfill\begin{tikzcd}[column sep = large]
                \Gamma.A.B \ar[d, disp] \ar[r, "\pair_{A,B}", dashed] &
                \Gamma.\Sigma_AB \ar[d, disp] \\
                \Gamma.A \ar[r, disp] &
                \Gamma,
            \end{tikzcd}\hfill\]
            \item[E.] for every type \(C\) in context \(\Gamma.\Sigma_AB\) and every term \(d\) of type \(C[\pair]\) a term \(\ind^\Sigma_d\) of \(C\): \[\hfill\begin{tikzcd}[column sep = large]
                \Gamma.A.B.C[\pair] \ar[d,disp]\ar[r,"\pair^\liltriangle"]\ar[dr, "\lrcorner"{anchor=center, pos=0.125}, draw=none] & \Gamma.\Sigma_AB.C \ar[d, disp] \\
                \Gamma.A.B \ar[u, bend left, "d"] \ar[r,"\pair"'] & \Gamma.\Sigma_AB, \ar[u, bend right, "\ind^\Sigma_d"', dashed]
            \end{tikzcd}\hfill\]
            \item[\(\beta\).] for every \(C\) and \(d\) as before, \(\ind^\Sigma_d[\pair]\) is propositionally equal to \(d\), that is, we have a term \(\beta^\Sigma_d\) of type \(\Id_{C[\pair]}[\ind^\Sigma_d[\pair],d]\).
        \end{description}
        We call a \(\Sigma\)-type \(\Sigma_AB\) \emph{intensional} if it satisfies the \(\beta\)-rule: for every \(C\) and \(d\) we have that \(\beta^\Sigma_d=\refl[d]\), and therefore that \(\ind^\Sigma_d[\pair]=d\).
        We call the \(\Sigma\)-type \emph{extensional} if it also satisfies the \(\eta\)-rule: for every \(C\) and \(c\) of type \(C\) we have \(\ind^\Sigma_{c[\pair]}=c\).
    \end{definition}
    In a strict display map category, we call a choice of of \(\Sigma\)-types \emph{(strictly) stable} if for every type $A$ in context $\Gamma$, every type \(B\) in context \(\Gamma.A\) and every morphism $\sigma:\Delta\to\Gamma$, we have: \begin{align*}
        (\Sigma_AB)[\sigma]&=\Sigma_{A[\sigma]}(B[\sigma^\liltriangle]), &
        \ind^\Sigma_d[\sigma^\liltriangle]&=\ind^\Sigma_{d[\sigma^\liltriangle]}, \\
        \pair_{A,B}[\sigma]&=\pair_{A[\sigma],B[\sigma^\liltriangle]}, &
        \beta^\Sigma_d[\sigma^\liltriangle]&=\makesize[r]\beta\ind^\Sigma_{d[\sigma^\liltriangle]}.
    \end{align*}
    We say that the \(\Sigma\)-types are \emph{weakly stable} if for every type $A$ in context $\Gamma$, every type \(B\) in context \(\Gamma.A\), and every morphism $\sigma:\Delta\to\Gamma$ we can equip \(\Sigma_A[\sigma]\) and \(\pair[\sigma]\) with the additional data required to form a \(\Sigma\)-type.

    \subsection{\texorpdfstring{\(\Pi\)}{Π}-types}
    
    The rules for axiomatic \(\Pi\)-types are: \begin{align*}
        &\prfstack[r]{\(\Pi\)F,}
            {\Gamma,x:A\vdash B\type}
            {\Gamma\vdash\Pi_{x:A}B\type}\, &
        &\prfstack[r]{\(\Pi\)I,}
            {\Gamma,x:A\vdash b:B}
            {\Gamma\vdash\lambda_{x:A}b:\Pi_{x:A}B} \\[2ex]
        &\prfstack[r]{\(\Pi\)E,}
            {\Gamma\vdash f:\Pi_{x:A}B}
            {\Gamma\vdash a:A}
            {\Gamma\vdash f\,a:B[a/x]}\, &
        &\prfstack[r]{\(\Pi\)\(\beta^{\sf{axi}}\).}
            {\Gamma,x:A\vdash b:B}
            {\Gamma\vdash a:A}
            {\Gamma\vdash\beta^\Pi_a(b):(\lambda_{x:A}b)\,a=b[a/x]}
    \end{align*}
    This induces a notion in display map categories:
    \begin{definition}[\(\Pi\)-types]
        We say that a display map category has \emph{axiomatic $\Pi$-types} if, for every type $A$ in context $\Gamma$ and type \(B\) in context \(\Gamma.A\), there exists:
        \begin{description}
            \item[F.] a type $\Pi_AB$ in context $\Gamma$.
            \item[I.] for every term \(b\) of type \(B\) in context \(\Gamma.A\) a term \(\lambda b\) of type \(\Pi_AB\) in context \(\Gamma\): \[\begin{tikzcd}[column sep = large]
                \Gamma.A.B \ar[d, disp] &
                \Gamma.\Pi_AB \ar[d, disp] \\
                \Gamma.A \ar[r, disp] \ar[u,"b", bend left] &
                \Gamma, \ar[u, bend right, dashed, "\lambda b"']
            \end{tikzcd}\]
            \item[E.] a term \(\sf{app}_{A,B}\) of type \(B\) for \(p_{\Pi_AB}^\weak:\Gamma.\Pi_AB.A^\weak\to\Gamma.A\): \[\begin{tikzcd}[column sep = large]
                \Gamma.\Pi_AB.A^\liltriangle \ar[dr, "p_{\Pi_AB}^\weak"']\ar[rr, "\sf{app}_{A,B}"] && \Gamma.A.B \ar[dl, disp] \\
                & \Gamma.A,
            \end{tikzcd}\]
            \item[\(\beta\).] for any term \(b\) of type \(B\), \(\sf{app}\circ(\lambda b)^\liltriangle\) is propositionally equal to \(b\), that is, we have a term \(\beta^\Pi_b\) of type \(\Id_B[\sf{app}\circ(\lambda b)^\liltriangle,b]\).
        \end{description}
        We call a \(\Pi\)-type is \emph{intensional} if it satisfies the \(\beta\)-rule: for every \(b\) we have that \(\beta^\Pi_b=\refl[b]\), and therefore that \(\sf{app}\circ(\lambda b)^\liltriangle=b\).
        We call the \(\Pi\)-type \emph{extensional} if it also satisfies the \(\eta\)-rule: for every \(f\) of \(\Pi_AB\) in context \(\Gamma\) we have \(\lambda(\sf{app}\circ f^\liltriangle)=f\).
    \end{definition}
    In a strict display map category, we call a choice of of \(\Pi\)-types \emph{strictly stable} under our choices for substitutions if for every \(B\) over \(\Gamma.A\) and \(\sigma:\Delta\to\Gamma\): \begin{align*}
        (\Pi_AB)[\sigma]&=\Pi_{A[\sigma]}(B[\sigma^\liltriangle]), &
        \sf{app}_{A,B}[\sigma^\liltriangle]&=\sf{app}_{A[\sigma],B[\sigma^\liltriangle]}, \\
        (\lambda b)[\sigma]&=\lambda(b[\sigma^\liltriangle]), &
        \beta^\Pi_d[\sigma^\liltriangle]&=\beta^\Pi_{b[\sigma^\liltriangle]}.
    \end{align*}
    We say that the \(\Pi\)-types are \emph{weakly stable} if for any \(B\) over $\Gamma.A$ and $\sigma:\Delta\to\Gamma$ we can equip $(\Pi_AB)[\sigma]$ and \(\sf{app}_{A,B}[\sigma^\liltriangle]\) with the additional data required to form a \(\Pi\)-type.
    
    In addition, we consider the \emph{(dependent) function extensionality axiom} stating for every \(f,f':\Pi_{x:A}B\) that the canonical function \((f=f)\to\Pi_{x:A}(f\,x=f'\,x)\) is an equivalence \cite{garner_strength_2009,lumsdaine_strong_2011}.
    This is where our naming convention of axiomatic/intensional/extensional type formers can be a bit confusing: the axiom of function extensionality is independent from \(\Pi\)-types being extensional (satisfying \(\beta\)- and \(\eta\)-rules).
    However, if the type theory is extensional --- every type former is extensional or equivalently the =-types are extensional --- then the axiom of function extensionality holds.

    \section{Additional details on the 2-isomorphism}\label{app:iso}

    In this appendix we present the ingredients involved in the statement of the 2-isomorphism of between the 2-categories \(\sf{PathCat}\) and \(\sf{DispCat}^{\sf{root}}_{=_{\sf{axi}},\bb1_{\sf{ext}},\Sigma_{\sf{ext}}}\). In \autoref{sec:dispcat} and \autoref{app:syntaxsemantics} we recalled the notion of a display map category and described how it can be equipped with the right structure to interpret a theory with specified type formers.
    Similarly, in \autoref{sec:pathcat} we outlined the notion of a path category. Here we expand on this by giving a detailed account of the 2-categories involved.

    \subparagraph*{Display map categories.}
    The 2-category $\sf{DispCat}$ consists of: \begin{description}
        \item[0-cells.] The display map categories.
        
        \item[1-cells.] If $\mathcal{C}$ and $\mathcal{D}$ are display map categories, the 1-cells from $\mathcal{C}$ to $\mathcal{C}'$ are functors $G:\mathcal{C}\to\mathcal{C}'$ that map display maps to display maps and pullback squares of display maps into pullback squares of display maps.
        
        \item[2-cells.] If $G$ and $G'$ are parallel 1-cells from $\mathcal{C}$ to $\mathcal{D}$, the 2-cells from $G$ to $G'$ are natural transformations $G\Rightarrow G'$.
    \end{description}
    This can be specialised to several suc-2-categories where the 0-cells are endowed with additional weakly stable type formers, and the 1-cells weakly preserve them.
    For example, we write $\sf{DispCat}_{=_\sf{axi}}$/$\sf{DispCat}_{=_\sf{int}}$/$\sf{DispCat}_{=_\sf{ext}}$ for the 2-categories with: \begin{description}
        \item[0-cells.] The display map categories equipped with weakly stable based axiomatic/intensional/extensional $=$-types.
        \item[1-cells.] If $\cal C$ and $\cal C'$ are 0-cells then the 1-cells from $\cal C$ to $\cal C'$ are those 1-cells of $\sf{DispCat}$ that weakly preserve the based axiomatic/intensional/extensional $=$-types --- see \autoref{sec:dispcat}.
        \item[2-cells.] The ones of $\sf{DispCat}$.
    \end{description}
    We follow the same pattern for other type formers (\(\bb1\), \(\Sigma\), \(\Pi\)).
    We write $\sf{DispCat}^{\sf{root}}$ for the sub-2-category where the 0-cells are rooted, and the 1-cells preserve the terminal object.
    Similarly, $\sf{DispCat}^{\sf{LF}}$ is the sub-2-category where the 0-cells satisfy the logical framework (LF) condition, and the 1-cells preserve the products and dependent exponents present.
    All of these restrictions can be combined in various ways; for example \(\sf{DispCat}_{=_\sf{axi},\bb1_\sf{ext},\Sigma_\sf{ext}}^{\sf{root}}\).
    
    Lastly, we write $\sf{DispCat}^{\sf{structured}}$ and $\sf{DispCat}^{\sf{strict}}$ for the full sub-2-categories of $\sf{DispCat}$ spanned by structured and strict display map categories, respectively.
    If this is combined with type formers, then the 0-cells of $\sf{DispCat}^{\sf{structured}}_{\dots}$ also have choices for the type formers, but they are still only weakly stable under re-indexing, and only weakly preserved by 1-cells.
    The 0-cells of $\sf{DispCat}^{\sf{strict}}_{\dots}$ have choices for the type formers, and they are strictly stable under re-indexing, but the 1-cells only weakly preserve them.
    Note that these are the same choices that Clairambault and Dybjer \cite{MR3272793} make for their 2-categories.

    \subparagraph*{Path categories.}
    The 2-category \(\sf{PathCat}\) consists of: \begin{description}
        \item[0-cells.] Path categories.
        \item[1-cells.] Functors preserving fibrations, trivial fibrations, pullbacks of fibrations, and the terminal object (and therefore also equivalences and path objects \cite{Brown1973}).
        \item[2-cells.] Natural transformations.
    \end{description}
    Again, in $\sf{PathCat}^{\sf{LF}}$ the 1-cells preserve the products and dependent exponents present, while $\sf{PathCat}^{\sf{structured}}$ and $\sf{PathCat}^{\sf{strict}}$ are full sub-2-categories.

    \section{Proofs}

    This appendix contains complete proofs of the statements of \autoref{sec:iso}.
    
    \pathtodisp*

    \begin{proof}
        Proven the first sentence of the statement, the second follows by \autoref{prop:clanasdisp}. To prove the first sentence, we are going to show that: \emph{formation} and \emph{introduction} for a type $A$ in context $\Gamma$ are provided by a path object for the fibration $A\fib \Gamma$; \emph{elimination} and \emph{$\beta$-axiom} follow by lifting (\autoref{the:lifting}); weak stability follows from the fact that the re-indexing of a path object is a path object itself.
        
        If $A$ is a type in context $\Gamma$ then we show that a path object \(\Id_A\coloneqq P_\Gamma A\) in \(\cal C_\Gamma\) has the structure of a (based) axiomatic \(=\)-type, where \emph{formation} is given by $(s,t):\Id_A\fib A\times_\Gamma A=\Gamma.A.A^\liltriangle$ and \emph{introduction} is given by the unique term $\refl_A$ of $\Id_A[\delta_A]$ such that $\delta_A^\liltriangle \refl_A = r$.
        Now, given a term \(a\) of \(A\), a type \(C\) in context \(\Gamma.A.\Id_A[a^\liltriangle]\), and term \(d\) of \(C[a,\refl[a]]\), \emph{elimination} and \emph{computation} will follow by using the lifting theorem on the square:
        \[\hfill\begin{tikzcd}[column sep=huge,row sep=small]
        	\Gamma & {\Gamma.C[a,\refl_A[a]]} & {\Gamma.A.\Id_A[a^\liltriangle].C} \\
        	\\
        	{\Gamma.A.\Id_A[a^\liltriangle]} && {\Gamma.A.\Id_A[a^\liltriangle].}
        	\arrow["d", from=1-1, to=1-2]
        	\arrow["{(a,\refl_A[a])}"', from=1-1, to=3-1]
        	\arrow["{(a,\refl_A[a])^\liltriangle}", from=1-2, to=1-3]
        	\arrow[from=1-3, to=3-3, fib]
        	\arrow[equals, from=3-1, to=3-3]
        \end{tikzcd}\hfill\]
        Let $d'$ be the composition arrow $\Gamma\to\Gamma.A.\Id_A[a^\liltriangle].C$. We observe that the map \((a,\refl_A[a])=a^\liltriangle\refl_A[a]\) on the left is a weak equivalence since it is a section of $p_Ap_{\Id_A[a^\liltriangle]}$ which is a trivial fibration as it is the pullback of \(s\) --- a trivial fibration --- along $a$.
        By \autoref{the:lifting} there is a term $\ind_d^=$ of $C$ such that $\ind_d^=(a,\refl_A[a])\simeq_{\Gamma.A.\Id_A[a^\liltriangle]}d'$.
        In other words $(\ind_d^=(a,\refl_A[a]),d')$ is the pre-composition of $p_{\Id_C}$ via some arrow $H:\Gamma\to\Gamma.A.\Id_A[a^\liltriangle].C.C^\liltriangle.\Id_C$.
        By the universal property of $\Id_C[(a,\refl_A[a])^{\liltriangle\liltriangle}]$, the arrow $H$ factors uniquely as an arrow: $$H':\Gamma\to \Gamma.C[a,\refl[a]].C[(a,\refl_A[a])^\liltriangle].\Id_C[(a,\refl_A[a])^{\liltriangle\liltriangle}]$$ which yields: $$\Gamma\xrightarrow{  ( \ind_d^=[a,\refl[a]] , d )  }\Gamma.C[a,\refl[a]].C[(a,\refl_A[a])^\liltriangle]$$ by post-composition via $p_{\Id_C[(a,\refl_A[a])^{\liltriangle\liltriangle}]}$. Since the arrow $i p_{\Id_C[(a,\refl_A[a])^{\liltriangle\liltriangle}]}$, where $i$ is the canonical isomorphism: $$i:\Gamma.C[a,\refl[a]].C[(a,\refl_A[a])^\liltriangle]\iso\Gamma.C[a,\refl[a]].C[a,\refl[a]]^\liltriangle$$ can be proven to be a path object of $p_{C[a,\refl[a]]}$, there is an arrow: $$j:\Gamma.C[a,\refl[a]].C[(a,\refl_A[a])^\liltriangle].\Id_C[(a,\refl_A[a])^{\liltriangle\liltriangle}]\to\Gamma.C[a,\refl[a]].C[a,\refl[a]]^\liltriangle.\Id_{C[a,\refl[a]]}$$ over $i$. The arrow $jH'$ yields: $$\Gamma\xrightarrow{  ( \ind_d^=[a,\refl[a]] , d )  }\Gamma.C[a,\refl[a]].C[a,\refl[a]]^\liltriangle$$ by post-composition via $p_{\Id_{C[a,\refl[a]]}}$, hence factors as a term of type: $$\Id_{C[a,\refl[a]]}[ \ind_d^=[a,\refl[a]] , d ]$$ so the $\beta$-axiom is satisfied.
        
        Finally, if $A$ is a type in context $\Gamma$ then the re-indexing of a path object along an arrow $\sigma:\Delta \to \Gamma$ constitutes a path object of $\Delta.A[\sigma]\fib\Delta$ --- this follows from the fact that the re-indexing of a fibration is a fibration, that the re-indexing of a trivial fibration is a trivial fibration, and that a section of a trivial fibration is a weak equivalence.
        By the previous argument, it satisfies \textit{elimination} and \textit{$\beta$-axiom}.
        Therefore, the axiomatic $=$-types that we defined are weakly stable.
    \end{proof}

    \identityvariations*
    
    \begin{proof}[Proof of \autoref{prop:identityvariations}.] The proof of the first equivalence is a straighforward adaptation of Bocquet's argument \cite[Section 3]{Bocquet2020} to non strict models.
        
        For the second equivalence, we show that a parametrized unbased axiomatic $=$-type for a type $A$ in a given context $\Gamma$ can be defined as a specific re-indexing of a parametrized based axiomatic $=$-type for (a choice of) $A^\liltriangle$ over $\Gamma.A$ --- namely, along the corresponding morphism $\delta_A^\liltriangle$.
        
        Let $A$ be a type over $\Gamma$. Let $(\Id_{A^\liltriangle},\refl_{A^\liltriangle},\ind^=,\beta^=)$ be a based axiomatic $=$-type for (a choice of) $A^\liltriangle$ over $\Gamma.A$. Then for the morphism \(\delta_A^\liltriangle:\Gamma.A.A^\liltriangle\to\Gamma.A.A^\liltriangle.A^\liltriangle[p_{A^\liltriangle}]\) --- intuitively sending \((\gamma,x,x')\) to \((\gamma,x,x,x')\) --- the type $\underline{\Id}_A\coloneqq\Id_{A^\liltriangle}[\delta_A^\liltriangle]$ in context $\Gamma.A.A^\liltriangle$ together with the term: $${\underline{\refl}_A\coloneqq(\Gamma.A\xrightarrow{\refl_{A^\liltriangle}[\delta_A]}\Gamma.A.\Id_{A^\liltriangle}[\delta_{A^\liltriangle}][\delta_A]\iso\Gamma.A.\Id_{A^\liltriangle}[\delta_A^\liltriangle][\delta_A])}$$ of type $\underline{\Id}_A[\delta_A]=\Id_{A^\liltriangle}[\delta_A^\liltriangle][\delta_A]$ can be equipped with the structure of an unbased axiomatic $=$-type as follows.\footnote{Observe that $(\underline{\Id},\underline{\refl})$ coincides with $(\Id,\refl)$ if \((\mathcal{C},\mathcal{T}, F, p)\) is split.}

        To validate the unbased axiomatic eliminator, whenever we are given a type $C$ in context $\Gamma.A.A^\liltriangle.\underline{\Id}_A$ and a term $d$ of type $C[\delta_A,\underline{\refl}_A]$, we have to construct a term \(\underline{\sf{ind}}_d^=\) of type \(C\) and a term \(\underline{\makesize\beta m}^=_d\) of type \(\Id_{C[\delta_A,\underline{\refl}_A]}[\sf{ind}_d^=[\delta_A,\underline{\refl}_A],d]\). However, if $C$ is a type in context $\Gamma.A.A^\liltriangle.\underline{\Id}_A$ and $d$ is a term of type $C[\delta_A,\underline{\refl}_A]$ then the arrow: $$\begin{tikzcd}
        	{\Gamma.A} &&&& {\Gamma.A.A^\liltriangle.\underline{\Id}_A}
        	\arrow["{\delta_A,\underline{\refl}_A}", from=1-1, to=1-5]
        \end{tikzcd}$$ coincides with the arrow: $$\begin{tikzcd}
        	{\Gamma.A} &&& && {\Gamma.A.A^\liltriangle.\Id_{A^\liltriangle}[\delta_A^\liltriangle]}
        	\arrow["{(\delta_A,\refl_{A^\liltriangle}[\delta_A])}", from=1-1, to=1-6]
        \end{tikzcd}$$ which is precisely the arrow $(a,\refl_A[a])$ in the based eliminator for $A$ and $a$ replaced by $A^\liltriangle$ and $\delta_A$ respectively.
        Therefore, for such $C$ and $d$, the based eliminator and its $\beta$-axiom collapse precisely into the ordinary, unbased, elimination and $\beta$-axiom for $(\underline{\Id},\underline{\refl})$.
        Via the same argument, one can verify that the parametrized unbased elimination and its $\beta$-axiom hold as well: if $C$ is a type in some extended context $\Gamma.A.A^\liltriangle.\underline{\Id}_A.\Omega$ and $d$ is a term of type $C[(\delta_A,\underline{\refl}_A)^\liltriangle]$ then there exists a choice of a parametrized unbased elimination and $\beta$-axiom: as before, they are specific instances of the parametric based elimination and $\beta$-axiom.
        
        Weak stability for $(\underline{\Id},\underline{\refl})$ follows from weak stability for $(\Id,\refl)$.

    The strategy to prove that transport is not admissible in a theory with only the Martin-L\"of elimination and computation rules and no \(\Pi\)-types is due to Rafa\"el Bocquet \cite[page 15]{Bocquet2020}. We expand the proof while adjusting a small subtlety surrounding symmetry. We also refer the reader to \cite{MR2469279}.
        We consider the syntactic category of a type theory with only the structural rules, the rules \(=\)F and \(=\)I, primitive types \(R\) and \(x:R\vdash S(x)\), and primitive terms \(r,r':R\), \(\rho:r=_Rr'\), \(\rho^{-1}:r'=_Rr\), \(s:S(r)\).
        We can visualise this as follows:
            \[\hfill
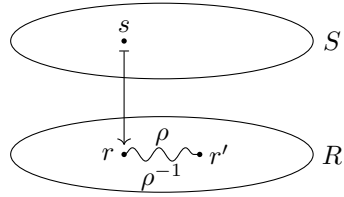
\begin{tikzpicture}[scale = 0.5]
                \fill (-1,3) node[above] {\(s\)}  circle[radius = 2pt];
                \fill (-1,0) node[left]  {\(r\)}  circle[radius = 2pt];
                \fill (1,0)  node[right] {\(r'\)} circle[radius = 2pt];
                \draw[decorate, decoration=snake] (-1,0) -- node[below] {\(\rho^{-1}\)} node[above] {\(\rho\)} (1,0);
                \draw[|->] (-1,2.75) -- (-1,0.25);
                \draw (0,3) ellipse (4 and 1);
                \node at (4.5,3) {\(S\)};
                \draw (0,0) ellipse (4 and 1);
                \node at (4.5,0) {\(R\)};
            \end{tikzpicture}\hfill\]
        The only terms are the variables, the generators \(r,r',\rho,\rho^{-1},s\), and the iterated reflexivity terms.
        Therefore, we see that \(S(r')\) is not inhabited in the empty context, which means that the model does not satisfy transport.
        However, we will show that the Martin-L\"of elimination and computation rules are admissible.
        Suppose that the assumptions of the elimination rule hold.
        If \(C\) does not depend on \(x',\chi\), then we can just take \(\sf{ind}^=_d(\alpha)\coloneqq d[a/x]\).
        Otherwise, since such a term \(d\) does not exist for \(C\coloneqq S(x')\), \(C\) must be of the form \(t'=_Tt''\) and \(d\) of the form \(\refl(t)\), where \(t'[x/x',\refl(x)/\chi]\equiv t\equiv t''[x/x',\refl(x)/\chi]\).
        
        If \(t'\equiv t''\), then we can take \(\sf{ind}^=_d(\alpha)\coloneqq\refl(t'[a/x,a'/x',\alpha/\chi])\).
        Otherwise, we see that \(t'\)~and~\(t''\) are both not generators, and, because they are terms of the same type \(T\), not reflexivity terms.
        The only remaining options are for \(t'\)~and~\(t''\) to be \(x\)~and~\(x'\), or \(x'\)~and~\(x\), respectively.
        In these remaining cases, we take \(\sf{ind}^=_d(\alpha)\coloneqq\alpha\) and \(\sf{ind}^=_d(\alpha)\coloneqq\alpha^{-1}\), respectively, where \(\alpha^{-1}\) is defined by \(\refl(b)^{-1}\coloneqq\refl(b)\), \((\rho)^{-1}\coloneqq\rho^{-1}\), and \((\rho^{-1})^{-1}\coloneqq\rho\).
        Our definition of \(\ind^=_d(\alpha)\) satisfies the \(\beta\)-rule because for all cases we see that \(\ind^=_d(\refl(a))\equiv d[a/x]\). \end{proof}
    
    \disptopath*

    \begin{proof}
        In \autoref{sec:iso} we explained that to have a rooted display map category with weakly stable (unparametrized based) axiomatic $=$-types and extensional \(\bb1\)- and \(\Sigma\)-types is to have a path category. 
        In this appendix we provide full verification of the fact that, for the clan $\mathcal{C}$ with weakly stable parametrized unbased axiomatic \(=\)-types induced by such a display map category, the property (\ref{equ:logicalequivalence}) holds.
    
        Let $\Delta\xrightarrow{\sigma}\Gamma$ be any morphism of contexts. We are left to prove that arrows as in (\ref{equ:logicalequivalence}) --- see \autoref{sec:iso} --- exist for every fibration $\Omega\fib\Gamma$ of $\cal C$. However, a fibration of $\cal C$ is nothing but a display map. Therefore, we are left to prove that arrows as in (\ref{equ:logicalequivalence}) exist for a type $A$ in context $\Gamma$, i.e. that there are arrows making the two squares:
        \[\begin{tikzcd}[column sep=tiny, row sep=small]	{\Delta.A[\sigma].A[\sigma]^\liltriangle.\underline{\Id}_{A[\sigma]}} && {\Delta.A[\sigma].A^\liltriangle[\sigma^\liltriangle].\underline{\Id}_A[\sigma^{\liltriangle\liltriangle}]} \\
        	\\
        	{\Delta.A[\sigma].A[\sigma]^\liltriangle} && {\Delta.A[\sigma].A^\liltriangle[\sigma^\liltriangle]}
        	\arrow[shift right, from=1-1, to=1-3]
        	\arrow[from=1-1, to=3-1]
        	\arrow[shift right, from=1-3, to=1-1]
        	\arrow[from=1-3, to=3-3]
        	\arrow[iso, from=3-1, to=3-3]
        \end{tikzcd}\] commute. In order to build ${\Delta.A[\sigma].A[\sigma]^\liltriangle.\underline{\Id}_{A[\sigma]}} \to {\Delta.A[\sigma].A^\liltriangle[\sigma^\liltriangle].\underline{\Id}_A[\sigma^{\liltriangle\liltriangle}]}$, we are going to use that $\underline{\Id}_{A[\sigma]}$ is an axiomatic $=$-type for $A[\sigma]$. By weak stability, so is $\underline{\Id}_A[\sigma^{\liltriangle\liltriangle}]$, therefore the same argument allows to build ${\Delta.A[\sigma].A^\liltriangle[\sigma^\liltriangle].\underline{\Id}_A[\sigma^{\liltriangle\liltriangle}]} \to {\Delta.A[\sigma].A[\sigma]^\liltriangle.\underline{\Id}_{A[\sigma]}}$.
        
        Since the square on the left: 
        \[\begin{tikzcd}[column sep=tiny,row sep=small]
        	{\Delta.A[\sigma]} && {\Delta.A[\sigma].A[\sigma]^\liltriangle.\underline{\Id}_A[\sigma^{\liltriangle\liltriangle}][j]} && {\Delta.A[\sigma].A^\liltriangle[\sigma^\liltriangle].\underline{\Id}_A[\sigma^{\liltriangle\liltriangle}]} \\
        	\\
        	{\Delta.A[\sigma].A[\sigma]^\liltriangle.\underline{\Id}_{A[\sigma]}} && {\Delta.A[\sigma].A[\sigma]^\liltriangle} && {\Delta.A[\sigma].A^\liltriangle[\sigma^\liltriangle]}
        	\arrow["{\underline{r}_A[\sigma]}"{description}, from=1-1, to=1-3]
        	\arrow["{\underline{r}_{A[\sigma]}}"{description}, from=1-1, to=3-1]
        	\arrow["{\delta_{A[\sigma]}}"{description}, from=1-1, to=3-3]
        	\arrow["{j^\liltriangle}"{near start}, iso, from=1-3, to=1-5]
        	\arrow[from=1-3, to=3-3]
        	\arrow[from=1-5, to=3-5]
        	\arrow[from=3-1, to=3-3]
        	\arrow["{j}"{near start}, iso, from=3-3, to=3-5]
        	\arrow["\lrcorner"{anchor=center, pos=0.125}, draw=none, from=1-3, to=3-3]
        \end{tikzcd}\] commutes, we can factor $\underline{r}_A[\sigma]$ as: \[\Delta.A[\sigma]\to\Delta.A[\sigma].A[\sigma]^\liltriangle.\underline{\Id}_{A[\sigma]}.\underline{\Id}_A[\sigma^{\liltriangle\liltriangle}][j][p_{\underline{\Id}_{A[\sigma]}}]\xrightarrow{p_{\underline{\Id}_{A[\sigma]}}^\liltriangle}\Delta.A[\sigma].A[\sigma]^\liltriangle.\underline{\Id}_A[\sigma^{\liltriangle\liltriangle}][j]\] in such a way that: \[\Delta.A[\sigma]\to\Delta.A[\sigma].A[\sigma]^\liltriangle.\underline{\Id}_{A[\sigma]}.\underline{\Id}_A[\sigma^{\liltriangle\liltriangle}][j][p_{\underline{\Id}_{A[\sigma]}}]\xrightarrow{p_{\underline{\Id}_{A[\sigma]}}}\Delta.A[\sigma].A[\sigma]^\liltriangle.\underline{\Id}_{A[\sigma]}\] is $\underline{r}_{A[\sigma]}$. In other words, the square:
        \[\begin{tikzcd}[column sep=tiny,row sep=small]
        	{\Delta.A[\sigma]} & {\Delta.A[\sigma].A[\sigma]^\liltriangle.\underline{\Id}_{A[\sigma]}.\underline{\Id}_A[\sigma^{\liltriangle\liltriangle}][j][p_{\underline{\Id}_{A[\sigma]}}]} \\
        	\\
        	{\Delta.A[\sigma].A[\sigma]^\liltriangle.\underline{\Id}_{A[\sigma]}} & {\Delta.A[\sigma].A[\sigma]^\liltriangle.\underline{\Id}_{A[\sigma]}}
        	\arrow[from=1-1, to=1-2]
        	\arrow["{\underline{r}_{A[\sigma]}}"{description}, from=1-1, to=3-1]
        	\arrow[from=1-2, to=3-2]
        	\arrow[equal, from=3-1, to=3-2]
        \end{tikzcd}\] commutes, hence by elimination the display map $p_{\underline{\Id}_A[\sigma^{\liltriangle\liltriangle}][j][p_{\underline{\Id}_{A[\sigma]}}]}$ has a section $\text{sect}$. The composition: \[j^\liltriangle p_{\underline{\Id}_{A[\sigma]}}^\liltriangle \text{sect}\] is the desired arrow ${\Delta.A[\sigma].A[\sigma]^\liltriangle.\underline{\Id}_{A[\sigma]}} \to {\Delta.A[\sigma].A^\liltriangle[\sigma^\liltriangle].\underline{\Id}_A[\sigma^{\liltriangle\liltriangle}]}$.
    \end{proof}

    \pathasdisp*

    \begin{proof}
        \autoref{prop:pathtodisp} and \autoref{prop:disptopath} show that we have a bijection between the 0-cells.
        To be precise, display map categories with a root, weakly stable axiomatic =types, and extensional \(\bb1\)- and \(\Sigma\)-types are categories equipped with one subclass of maps (display maps), while path categories are equipped with two subclasses (fibrations and weak equivalences).
        However, because of \autoref{prop:clanasdisp} we know that the display maps coincide with the fibrations of the display map category, and the class of weak equivalences is uniquely specified by saturation (\autoref{the:saturation}) to be the homotopy equivalences.
        So, this bijection only forgets or reconstructs the uniquely specified class of weak equivalences.
        
        Therefore, the only thing we need to check is that the 1-cells are the same, as the 2-cells are chosen to be all natural transformations for both.
        If we have a 1-cell of path categories, then it preserves pullbacks of display maps as they coincide with the fibrations, and it weakly preserves =-types because every =-type is a path object.
        It also pseudo preserves \(\bb1\)- and \(\Sigma\)-types because they are given by up to isomorphism by identity and composition.
        Hence it is also a 1-cell between such display map categories.
        Conversely, if we are given a 1-cell between such display map categories, then it preserves fibrations of the corresponing path categorical structures --- as these are the display maps --- pullbacks of fibrations, and the terminal object. Hence we are left to prove that it preserves weak equivalences of the corresponing path categorical structures. By \autoref{the:saturation} --- see also the proof of \autoref{prop:disptopath} --- the latter are precisely the homotopy equivalences definable in the given domain and codomain display map categories. As explained in \autoref{sec:dispcat}, such a homotopy equivalence consist of a term \(b:\Gamma.A\to\Gamma.B\) of \(B\) for \(p_A\) such that there exist a term \(a:\Gamma.B\to\Gamma.A\) of \(A\) for \(p_B\) and terms of \(\Id_A[ab,1_{\Gamma.A}]\) and \(\Id_B[ba,1_{\Gamma.B}]\). Then, since the given 1-cell weakly preserves axiomatic $=$-types --- i.e. up to homotopy equivalence --- and maps pullback squares of display maps to pullback squares, if follows that the image of $b$ is itself a homotopy equivalence. This concludes the argument.
    \end{proof}
    
    \section{Different categorical frameworks for dependent type theory}\label{app:framework}

    We have used display map categories as our categorical framework for the basic structural rules of dependent type theory; however, there are many other frameworks, which are mostly biequivalent as 2-categories (for an overview see \cite{Ahrens_Lumsdaine_North_2025}).
    Because we cite work using other frameworks --- most notably the work of Clairambault and Dybjer \cite{MR3272793} and Bocquet \cite{Bocquet2020,MR4481908} using categories with families, and the work of Lumsdaine and Warren \cite{MR3372323} using full comprehension categories --- we will clarify some subtleties.
    
    Categories with families have primitive notions for types and terms, and --- because they require substitution to be strictly functorial --- correspond to a strict notion of semantics.
    However, although this means that substitution respects identity and composition on the nose, we can still consider type formers that are only weakly stable, as is done by Bocquet \cite{MR4481908}.
    Full comprehension categories make the simplification of identifying terms with sections of display maps (which are always in bijection for categories with families).
    They consist of a full and faithful functor \(\cal T\to\cal C^\to\) (sending every type to its display map), such that the composition with \(\sf{codom}:\cal C^\to\to\cal C\) is a Grothendieck fibration, and \(p:\cal T\to\cal C^\to\) preserves Cartesian morphisms.
    This is a weak notion of semantics, meaning that there can be multiple isomorphic options for the re-indexing.
    In a display map category we simplify further by identifying types with their display map.
    This means that \(\cal T\to\cal C^\to\) is an inclusion of categories, so it is also injective on objects.
    In addition, we ask for this subclass to be replete.
    This is not always true for a full comprehension category, but this notion is biequivalent.
    We can think of a display maps in a comprehension category as strict display maps, while the display maps of a display map category are pseudo display maps.

    \subparagraph*{Structured and strict.}
    If we ask for the Grothendieck fibration to be cloven or split we obtain structured and strict notions.
    The structured display map categories are biequivalent to the cloven comprehension categories, where the strict display maps correspond to the display maps of the comprehension category, and the more general display maps are their repletion.
    This is true because we do not require the 1-cells between structured display map categories to preserve the choice of re-indexing on the nose, so we get a closer correspondence than Ahrens, Lumsdaine, and North \cite{Ahrens_Lumsdaine_North_2025}.
    Here we have to be a bit careful: multiple types can be sent to the same display map, so if we just replace \(\cal T\) by the image of \(p:\cal T\to\cal C^\to\), then having choices for re-indexing types generally does not even give choices for re-indexing display maps.
    Instead, we should replace \(\cal C\) by the \emph{mapping cylinder} \(\cal M\) of \[\hfill\begin{tikzcd}
        \cal T \ar[r, "p"] & \cal C^\to \ar[r,"\sf{dom}"] & \cal C.
    \end{tikzcd}\hfill\] 
    \(\cal M\) is given by the disjoint union of \(\cal C\) and \(\cal T\) on objects.
    We specify the morphisms by saying that both are full subcategories, and maps between a type \(A\) in \(\cal T\) and a context \(\Delta\) in \(\cal C\) or vice versa are given by maps between the context \(\Gamma.A\coloneqq\sf{dom}(p_A)\) and the context \(\Delta\).
    Now we see that 
    The strict display maps are now the maps \(A\to\Gamma\) given by \(p_A\) for a type \(A\) in \(\cal T\) and its context \(\Gamma\coloneqq\sf{codom}(p_A)\).
    So, because this uniquely corresponds to an object of \(\cal T\), we still have good choices for re-indexing.
    In addition, the strict display map categories are biequivalent to the strict comprehension categories.

    \subsection{Strictifying}

    We explain how to extend type formers from strict display maps to general display maps, in a structured display map category. This is immediate for extensional $\bb1$- and $\Sigma$-types. Regarding axiomatic $=$-types, we first observe that in this case a general display map is a strict display map, or an isomorphism, or a strict display map precomposed by an isomorphism --- this follows immediately by using extensional $\bb1$- and $\Sigma$-types. For strict display maps, we are done. Let us examine the other two cases:
    
    (a) For an isomorphism $\Omega \iso \Gamma$, a re-indexing $\Omega\times_\Gamma\Omega\fib\Omega$ of $\Omega \iso \Gamma$ along itself is represented by the identity $\Omega = \Omega$, and in this case $(1,1):\Omega \to \Omega\times_\Gamma\Omega$ is $\Omega = \Omega$. Taking both \textit{formation} and \textit{introduction} to be the identity arrow $\Omega = \Omega$, elimination and $\beta$-axiom ($\beta$-rule) are satisfied.
    
    (b) For a general display map of the form $\Omega \iso \Gamma.A \disp \Gamma$, a re-indexing of $\Omega \iso \Gamma.A \disp \Gamma$ along itself is $\Omega.A[P_Af]\disp\Omega$, where $f$ is the isomorphism $\Omega \iso \Gamma.A$, and $\Omega.A[P_Af]$ is isomorphic to $\Gamma.A.A^\liltriangle$. Let $i$ be this isomorphism. The factorisation: \[\begin{tikzcd}
    	{\Gamma.A} && {\Gamma.A.A^\liltriangle.\underline{\Id}_A} && {\Gamma.A.A^\liltriangle}
    	\arrow["{\delta_A^\liltriangle\underline{\refl}_A}", from=1-1, to=1-3]
    	\arrow[from=1-3, to=1-5, disp]
    \end{tikzcd}\] of $\delta_A$ induces a factorisation: \[\begin{tikzcd}
    	{\Omega\cong\Gamma.A} && {\Gamma.A.A^\liltriangle.\underline{\Id}_A\cong\Omega.A[P_Af].\underline{\Id}_A[i]} && {\Omega.A[P_Af]}
    	\arrow["{\delta_A^\liltriangle\underline{\refl}_A}", from=1-1, to=1-3]
    	\arrow[from=1-3, to=1-5,disp]
    \end{tikzcd}\] of $\Omega\to\Omega.A[P_Af]$ which satisfies elimination and $\beta$-axiom.

    \structuredpathtomodel*
    \begin{proof}
        As explained in \cite{Ahrens_Lumsdaine_North_2025}, the (structured) display map category $\cal C$ --- see \autoref{the:pathasdisp} --- can be phrased as a structured comprehension category, over $\cal C$, with the same type formers. Then, by \cite[Theorem 3.4.1]{MR3372323}, the left adjoint splitting of the latter forms a split rooted LF comprehension category, again over $\cal C$, with stable extensional $\bb1$- and $\Sigma$-types and weakly stable axiomatic $=$-types. By \cite[Proposition 25 and Theorem 1]{MR4481908}, applying the left adjoint splitting again to the latter yields a strict rooted LF display map category, again over $\cal C$, this time with stable type formers. Its induced class of display maps constitute the strict display maps of a strict display map category --- whose general display maps are obtained by repleting this class --- as in the statement: it is immediate to extend type formers to general display maps as above.
    \end{proof}

    \pathcatcoherence*

    \proof Let $\mathcal{C}$ be a structured LF path category. By \autoref{the:pathasdisp} the model $\mathcal{C}_!$ of extensional $\bb1$- and $\Sigma$-types and axiomatic $=$-types provided by \autoref{the:structuredpathtomodel} constitutes another structure of an LF path category over the category $\mathcal{C}$, which we continue denoting as $\mathcal{C}_!$.
    In particular $\mathcal{C}_!$ is a structured and strict path category if we prove that by repleting strict display maps, we obtain a family of fibrations.
    This is because by \autoref{the:pathasdisp} the path categories $\mathcal{C}$ and $\mathcal{C}_!$ are equivalent in $\sf{PathCat}$. 
    
    So, we prove that, we obtain a family of fibrations for a path categorial structure, i.e. a clan:
    \begin{enumerate}
        \item Let $\Gamma$ be a context. Being $\mathcal{C}_!$ rooted $\Gamma \cong 1.A_1.A_2.\cdots.A_n$ for some --- possibly null --- natural $n$ and some type $A_i$ in context $1.A_1.  \dots . A_{i-1}$ for every $i= 1, \dots,n$. Since $(\mathcal{C},\mathcal{T},F,p)$ has $\Sigma$-types with $\beta,\eta$, $1.A_1.  \dots .A_n\cong 1.\Sigma_{A_1}\Sigma_{A_2}  \dots \Sigma_{A_{n-1}}A_n$ hence we factored $\Gamma\to 1$ as $\Gamma \iso 1.A_1.A_2.  \dots .A_n\iso 1.\Sigma_{A_1}\Sigma_{A_2}  \dots \Sigma_{A_{n-1}}A_n \disp 1$ which is a fibration, according to our notion.
        
        \item Let $\Omega\fib \Delta\fib \Gamma$ be fibrations. Let us verify that their composition is a fibration as well in the possible four cases, (i), (ii), (iii) and (iv).
        
        (i) If $\Omega\iso \Delta$ is an isomorphism, then we are done.
        
        (ii) If $\Delta\fib \Gamma$ is a display map $\Gamma.A\disp\Gamma$ and $\Omega \fib \Delta$ is either of the form $\Omega=\Delta.B\disp\Delta$ or $\Omega\iso\Delta.B\disp\Delta$ then there exists a commutative diagram of the form: \[\begin{tikzcd}[column sep=small, row sep=tiny]
        	\Omega && {\Gamma.A.B} && {\Gamma.\Sigma_AB} \\
        	\\
        	&& {\Gamma.A} && \Gamma
        	\arrow[iso, from=1-1, to=1-3]
        	\arrow[from=3-3, to=3-5]
        	\arrow[from=1-3, to=3-3]
        	\arrow[from=1-5, to=3-5]
        	\arrow[iso, from=1-3, to=1-5]
        \end{tikzcd}\] because $\mathcal{C}_!$ has extensional $\Sigma$-types, hence we are done.
        
        (iii) If $\Delta\fib\Gamma$ is of the form $\Delta\xrightarrow{f} \Gamma.A\disp\Gamma$, where $f$ is an isomorphism, and $\Omega \fib \Delta$ is either of the form $\Omega=\Delta.B\disp\Delta$ or $\Omega\iso\Delta.B\disp\Delta$ then the diagram: \[\begin{tikzcd}[column sep=small, row sep=tiny]
        	\Omega && {\Delta.B} && {\Gamma.A.B[f^{-1}]} \\
        	\\
        	&& \Delta && {\Gamma.A} && \Gamma
        	\arrow[from=1-3, to=3-3]
        	\arrow["{f^{-1}}", from=3-5, to=3-3]
        	\arrow[from=1-5, to=3-5]
        	\arrow[iso, from=1-3, to=1-5]
        	\arrow["\lrcorner"{anchor=center, pos=0.125, rotate=-90}, draw=none, from=1-5, to=3-3]
        	\arrow[iso, from=1-1, to=1-3]
        	\arrow[from=3-5, to=3-7]
        \end{tikzcd}\] commutes, hence $\Omega\fib\Delta\fib\Gamma$ equals $\Omega\iso\Delta.B\iso\Gamma.A.B[f^{-1}]\disp\Gamma.A\disp\Gamma$ and we are done by (ii).
        
        (iv) If $\Delta\iso\Gamma$ is an isomorphism and $\Omega \fib \Delta$ is either of the form $\Omega=\Delta.B\disp\Delta$ or $\Omega\iso\Delta.B\disp\Delta$ then we are done by following the same argument as (iii).
        
        \item The pullback of an isomorphism exists along every arrow and every pullback of an isomorphism along some arrow is an isomorphism, hence a fibration. Now, if a fibration $f$ is of the form $\Omega\iso \Gamma.A\disp\Gamma$ and $\sigma$ is any arrow $\Delta\to\Gamma$, then, the outer square:
        \[\begin{tikzcd}
        	{\Delta.A[\sigma]} && {\Gamma.A} & \Omega \\
        	&&& {\Gamma.A} \\
        	\\
        	\Delta &&& \Gamma
        	\arrow[from=1-1, to=1-3]
        	\arrow[from=1-1, to=2-4]
        	\arrow[from=1-1, to=4-1]
        	\arrow[iso, from=1-3, to=1-4]
        	\arrow[iso, from=1-4, to=2-4]
        	\arrow[from=2-4, to=4-4]
        	\arrow[""{name=0, anchor=center, inner sep=0}, "\sigma"', from=4-1, to=4-4]
        	\arrow["\lrcorner"{anchor=center, pos=0.125}, draw=none, from=1-1, to=0]
        \end{tikzcd}\] is a pullback square. Hence the display map $\Delta.A[\sigma]\to\Delta$ --- for some re-indexing of $A$ along $\sigma$ according to $\cal C_!$ --- is a pullback of $f$ along $\sigma$ and it is a fibration. Moreover, by the universal property of pullbacks, any other pullback of \( f \) along \( \sigma \) is equal to the precomposition of \( \Delta.A[\sigma] \to \Delta \) with an isomorphism, and thus is itself a fibration.
        
        \item Finally, by repletion, identities and isomorphisms are fibrations since $\mathcal{C}_!$ has extensional $\bb1$-types. \qedhere
        \end{enumerate}

        \endproof
\end{document}